\documentclass{amsart}
\usepackage{amssymb}
\usepackage{amsmath}
\usepackage[dvipsnames]{xcolor}
\usepackage{graphicx}
\usepackage{manfnt}
\usepackage{amsthm}
\usepackage[mathscr]{euscript}
\usepackage{mathtools}
\usepackage{epstopdf}
\usepackage{bigints}
\usepackage{enumerate}
\usepackage{enumitem}
\graphicspath{ {./images/} }
\usepackage{tikz-cd}
\usepackage{bm}
\usetikzlibrary{arrows}
\usetikzlibrary{positioning}
\usepackage{mathrsfs}
\numberwithin{equation}{section}
\usepackage{hyperref}
\usepackage{tikz-cd}
\usepackage{multicol}
\usepackage{mathrsfs} 
\usepackage{braket}

\newcommand{\bbm}{\begin{bmatrix}}
\newcommand{\ebm}{\end{bmatrix}}
\newcommand{\bv}{\begin{vmatrix}}
\newcommand{\ev}{\end{vmatrix}}

\newcommand{\g}{\mathfrak{g}}

\newcommand{\n}{\mathfrak{n}}
\newcommand{\h}{\mathfrak{h}}

\newcommand{\C}{\mathbb{C}}
\newcommand{\Z}{\mathbb{Z}}
\newcommand{\A}{\mathbb{A}}
\newcommand{\beq}{\begin{equation*}}
\newcommand{\eeq}{\end{equation*}}
\newcommand{\beqn}{\begin{eqnarray*}}
\newcommand{\eeqn}{\end{eqnarray*}}

\newcommand{\mf}{\mathfrak}
\newcommand{\mc}{\mathcal}
\newcommand{\bp}{\begin{pmatrix}}
\newcommand{\ep}{\end{pmatrix}}

\DeclareMathOperator{\End}{End}
\DeclareMathOperator{\Hom}{Hom}

\DeclareMathOperator{\im}{im }

\DeclareMathOperator{\Lie}{Lie}
\DeclareMathOperator{\Ext}{Ext}

\DeclareMathOperator{\SL}{SL}
\DeclareMathOperator{\coker}{coker}

\theoremstyle{plain}
\newtheorem{theorem}{Theorem}[section]
\newtheorem*{theorem*}{Theorem}
\theoremstyle{plain}
\theoremstyle{definition}
\newtheorem{definition}[theorem]{Definition}
\newtheorem{lemma}[theorem]{Lemma}
\theoremstyle{definition}

\theoremstyle{plain}

\theoremstyle{plain}

\theoremstyle{remark}
\newtheorem{remark}[theorem]{Remark}
\theoremstyle{definition}

\theoremstyle{definition}

\title{An example of the Jantzen filtration of a D-module}
\author{Simon Bohun and Anna Romanov}
\date{}

\begin{document}

\maketitle

\begin{abstract}
    We compute the Jantzen filtration of a $\mc{D}$-module on the flag variety of $\SL_2(\C)$. At each step in the computation, we illustrate the $\mf{sl}_2(\C)$-module structure on global sections to give an algebraic picture of this geometric computation. We conclude by showing that the Jantzen filtration on the $\mc{D}$-module agrees with the algebraic Jantzen filtration on its global sections, demonstrating a famous theorem of Beilinson--Bernstein. 
\end{abstract}
\tableofcontents
\section{Introduction}
\label{introduction}

\subsection{Overview}
\label{sec: overview}

Jantzen filtrations arise in many situations in representation theory. The Jantzen filtration of a Verma module over a semisimple Lie algebra provides information on characters (the Jantzen sum formula) \cite{Jantzen}, and gives representation-theoretic significance to coefficients of Kazhdan--Lusztig polynomials (the Jantzen conjectures) \cite{BBJantzen}. The Jantzen filtration of a Weyl module over a reductive algebraic group of positive characteristic is a helpful tool in the notoriously difficult problem of determining irreducible characters \cite{Jantzen}. Jantzen filtrations also play a critical role in the unitary algorithm of \cite{ALTV}, which determines the irreducible unitary representations of a real reductive group. 

Though the utility of Jantzen filtrations in applications is primarily algebraic (providing information about characters or multiplicities of representations), establishing deep properties of the Jantzen filtration usually requires a geometric incarnation due to Beilinson--Bernstein. In \cite{BBJantzen}, Beilinson--Bernstein introduce a $\mc{D}$-module version of the Jantzen filtration, which provides them with powerful geometric tools to analyze its structure. The constructions in \cite{BBJantzen} require technical and deep machinery in the theory of $\mc{D}$-modules, and as such, may not be easily accessible to a reader unfamiliar with this geometric approach to representation theory. However, the persistent utility of Beilinson--Bernstein's results indicates that the geometric Jantzen filtration is a critical tool. 

In our experience, it is often enlightening, insightful, and non-trivial to describe a difficult construction in a simple example. This is the purpose of this paper: to illustrate the construction of Beilinson--Bernstein in the simplest non-trivial example. In doing this, we include simplified proofs of Beilinson--Bernstein's results for the Lie algebra $\mf{sl}_2(\C)$ and detailed computations which do not appear in the original paper. 

The main contribution of our example is to provide algebraic insight into a fundamental geometric construction. Beilinson--Bernstein localisation is a powerful bridge between representation theory and algebraic geometry which has provided geometric proofs of several important algebraic theorems. This strategy of using geometric tools to approach algebraic problems is effective, but it has a drawback --- without deep knowledge of the geometry involved, the algebraist using these results is left without a sense of what is happening under the hood, and as a result, geometric results are often used as black boxes. 

Our approach in this paper is to shine light into the black box by providing a series of algebraic snapshots of a geometric computation. We do this by computing the global sections of the $\mc{D}$-modules which arise at each step in the computation and illustrating the corresponding $\mf{sl}_2(\C)$-representations. Here we mean ``illustrate'' in the most literal sense --- we include eight figures in which we draw precise pictures of these representations. Our hope is that by giving a concrete visual description, we are able to provide readers with algebraic intuition for the general construction.

This paper is concerned with the example of $\SL_2(\C)$. However, some amount of general theory will be helpful to set the scene. We dedicate the remainder of the introduction to  orienting the reader with the necessary general theory.

\subsection{The algebraic Jantzen filtration}
\label{sec: the algebraic Jantzen filtration intro}

 Let $\mf{g}\supset \mf{b} \supset \mf{h}$ be a complex semisimple Lie algebra, a Borel subalgebra, and a Cartan subalgebra. Denote by $\mf{n} = [\mf{b},\mf{b}]$ the nilradical of $\mf{b}$, and $\overline{\mf{b}}$ the opposite Borel subalgebra. Given a weight $\lambda \in \mf{h}^*$, let $M(\lambda) = \mc{U}(\mf{g}) \otimes_{\mc{U}(\mf{b})} \C_\lambda$ be the corresponding Verma module, $I(\lambda)$ the corresponding dual Verma module (defined to be the direct sum of the weight spaces in the $\g$-module $\Hom_{\mc{U}(\overline{\mf{b}})}(\mc{U}(\mf{g}), \C_\lambda)$), and 
\[
\psi: M(\lambda) \rightarrow I(\lambda).
\]
the canonical $\mf{g}$-module homomorphism from $M(\lambda)$ to $I(\lambda)$.

The algebraic Jantzen filtration of $M(\lambda)$ involves a deformation of the above set-up in a specified direction $\gamma \in \h^*$. The deformation is constructed as follows.
Given $\gamma \in \h^*$, let $T=\mc{O}(\C \gamma)$ be the ring of regular functions on the line $\C \gamma \subset \h^*$. This can be identified with a polynomial ring $\C[s]$. Denote by $A=T_{(s)}$ the local ring of $T$ at the prime ideal $(s)$. 

We use the ring $A$ to construct the corresponding {\em deformed Verma module}, defined to be the $(\mf{g},A)$-bimodule
\[
M_A(\lambda) := \mc{U}(\mf{g}) \otimes_{\mc{U}(\mf{b})} A_\lambda,
\]
where $A_\lambda=A$ is the $(\h, A)$ bimodule given by 
\begin{equation}
    \label{eq: deformation in direction gamma}
    h \cdot a = (\lambda(h) + \gamma(h)s)a
\end{equation}
for $h \in \h$, $a \in A$, extended trivially to $U(\mf{b})$. Equation \eqref{eq: deformation in direction gamma} demonstrates that $M_A(\lambda)$ is a ``deformation of $M(\lambda)$ in the direction $\gamma$''. 

Similarly, the {\em deformed dual Verma module} $I_A(\lambda)$ is defined to be the sum of deformed weight spaces (see \eqref{eq: deformed weight space}) in the $(\mf{g},A)$-bimodule
\[
\Hom_{\mc{U}(\overline{\mf{b}})}(\mc{U}(\mf{g}), A_\lambda).
\]
There is a canonical $(\mf{g},A)$-module homomorphism 
\begin{equation}
    \label{eq: canonical hom deformed intro}
    \psi_A: M_A(\lambda) \rightarrow I_A(\lambda).
\end{equation}
Setting $s=0$ recovers the usual Verma and dual Verma modules, and the canonical morphism $\psi$. 

The $A$-submodules $s^iM_A(\lambda)$ and $s^iI_A(\lambda)$ are $\mf{g}$-stable for all $i$, so both $M_A(\lambda)$ and $I_A(\lambda)$ have $(\mf{g},A)$-module filtrations given by powers of $s$. The Jantzen filtration of $M_A(\lambda)$ is the filtration obtained by pulling back the filtration of $I_A(\lambda)$ by powers of $s$ along the canonical homomorphism $\psi_A$ \eqref{eq: canonical hom deformed intro}. Setting $s=0$ recovers a filtration of $M(\lambda)$. This is the {\em algebraic Jantzen filtration}\footnote{Analogous constructions yield Jantzen filtrations of the Weyl modules and principal series representations mentioned in \S\ref{sec: overview} \cite{Jantzen, BBJantzen}. Because we focus on Verma modules in our example, we will not define these other Jantzen filtrations precisely.} of the Verma module $M(\lambda)$. 

\begin{remark}
    (Computability of Jantzen filtration) The algebraic Jantzen filtration is traditionally formulated in terms of a contravariant form, which explicitly realises the canonical map between $M(\lambda)$ and $I(\lambda)$. See, for example, \cite{Shapovalov, Jantzen}. This explicit realisation makes the filtration directly computable, which is useful in applications. In contrast, other important representation-theoretic filtrations, such as composition series, are known to exist, but are much more difficult to compute algorithmically. 
\end{remark}

For $\mf{g} = \mf{sl}_2(\C)$, the Jantzen filtration coincides with the composition series, as our computations in \S\ref{sec: example} illustrate. However, for larger Lie algebras (already starting at $\mf{sl}_3(\C)$), the Jantzen filtration differs from the composition series, and  carries fundamental information about Verma modules and related representations. Jantzen conjectured \cite[\S5.17]{Jantzen} that for $\gamma = \rho$ (the half-sum of positive roots) the Jantzen filtration satisfies the following properties: 
\begin{enumerate}
    \item Embeddings of Verma modules $M(\mu) \hookrightarrow M(\lambda)$ are strict for Jantzen filtrations. 
    \item The Jantzen filtration coincides with the socle filtration. In particular, the filtration layers are semisimple.
\end{enumerate}

Subsequent work by Barbasch \cite{Barbasch}, Gabber-Joseph \cite{GJ}, and others revealed that Jantzen's conjectures have deep consequences. In particular, Jantzen's conjectures imply a stronger version of Kazhdan--Lusztig's famous conjecture on composition series multiplicities of Verma modules \cite{KL}: multiplicities of simple modules in layers of the Jantzen filtration are given by coefficients of a corresponding Kazhdan--Lusztig polynomial. 

Kazhdan--Lusztig's original multiplicity conjecture was proven by Beilinson--Bernstein in \cite{BB81} using $\mc{D}$-module techniques. A proof of Jantzen's conjectures did not appear until 12 year later in \cite{BBJantzen}, using a significant extension of the geometric techniques used in \cite{BB81}. In the following section, we outline their approach.

\begin{remark}
    (Algebraic proof of Jantzen's conjectures)
    In \cite{Williamson}, Williamson provided an alternate proof of Jantzen's conjectures using Soergel bimodule techniques, following previous work of Soergel and K{\"u}bel \cite{Soergel08, Kubel1, Kubel2}. Williamson's proof holds for Verma modules, whereas Beilinson--Bernstein's proof also holds for more general Harish-Chandra modules. 
\end{remark}

\begin{remark}
\label{rem: deformation direction algebraic}
    (Deformation direction) The definition of the algebraic Jantzen filtration relies on a choice of deformation direction $\gamma \in \h^*$, which also has a geometric manifestation in Beilinson--Bernstein's construction. It is clear from the definitions that this direction should be non-degenerate; i.e. that it should not lie on any root hyperplanes. However, it was a long-standing problem (raised in \cite{BBJantzen}) as to whether the deformation direction need be dominant. Williamson showed in \cite{Williamson} that it does, giving examples of non-dominant deformation directions resulting in different filtrations for Lie algebras as small as  $\mf{g}=\mf{sl}_4(\C)$.
\end{remark}

\subsection{The geometric Jantzen filtration}
\label{sec: The geometric Jantzen filtration}

Beilinson--Bernstein's approach to the Jantzen conjectures is to relate the algebraic Jantzen filtration to a natural geometric filtration on the corresponding $\mc{D}$-module under Beilinson--Bernstein localisation. They then argue that this {\em geometric Jantzen filtration} coincides with the weight filtration on the $\mc{D}$-module, providing them access to powerful techniques in weight theory. In this section, we outline Beilinson--Bernstein's construction. More details can be found in \cite{BBJantzen}. 

\subsubsection{Monodromy filtrations}
\label{sec: monodromy filtrations}
Geometric Jantzen filtrations are intimately related to monodromy filtrations. Given an object $A$ in an abelian category $\mc{A}$ and a nilpotent endomorphism $s \in \End_{\mc{A}}(A)$, the {\em monodromy filtration} of $A$ is defined to be the unique increasing exhaustive filtration $\mu^\bullet$ on $A$ such that $s\mu^n \subset \mu^{n-2}$, and for $k \in \mathbb{N}$, $s^k$ induces an isomorphism $\mathrm{gr}_\mu^k A \simeq \mathrm{gr}_\mu^{-k} A.$ 

The monodromy filtration of $A$ induces a filtration $J_{!}^\bullet$ on $\ker s $ and a filtration $J_+^\bullet$ on $\coker s$ in the natural way. Moreover, on $\ker s$ and $\coker s $, the monodromy filtration can be described explicitly in terms of powers of $s$.  Namely,  
\begin{equation}
    \label{eq: geometric Jantzen filtration as powers of s} J_!^i = \ker s \cap \im s^{-i} \text{ and } J_+^i = (\ker s^{i+1} + \im s) / \im s, 
\end{equation}
where it is taken that $\im s^i = A$ for $i \leq 0$ and $\ker s^i=0$ for $i\leq 0$ \cite[\S4.1]{BBJantzen}. (See also \cite[\S1.6]{Del}.) 

\subsubsection{Geometric Jantzen filtrations}
\label{sec: geometric Jantzen filtrations}
Certain $\mc{D}$-modules come equipped with nilpotent endomorphisms, and thus acquire monodromy filtrations. In particular, the {\em maximal extension functor}\footnote{See \S \ref{sec: maximal extension} for the precise definition of this functor.} provides a recipe for constructing $\mc{D}$-modules with nilpotent endomorphisms from $\mc{D}$-modules on open subvarieties using a deformation procedure. 

More precisely, if $Y$ is a smooth algebraic variety with a fixed regular function $f: Y \rightarrow \mathbb{A}^1$, the maximal extension $\Xi_f \mc{M}_U$ of a holonomic $\mc{D}_U$-module $\mc{M}_U$ on $U=f^{-1}(\mathbb{A}^1 - \{0\})$ is constructed by deforming $\mc{M}_U$ by the ring $\C[s]/s^n$ using the function $f$, then pushing forward the deformed $\mc{M}_U$ along the inclusion map $j:U \hookrightarrow Y$. The resulting $\mc{D}_Y$-module is an object in the abelian category of holonomic $\mc{D}_Y$-modules which has a natural nilpotent endomorphism $s$ arising from the deformation of $\mc{M}_U$. Hence it has a monodromy filtration. 

The construction of the maximal extension functor guarantees that 
\[
\ker(s: \Xi_f \mc{M}_U \rightarrow \Xi_f \mc{M}_U) = j_! \mc{M}_U
\]
and
\[
\coker(s: \Xi_f \mc{M}_U \rightarrow \Xi_f \mc{M}_U)=j_+ \mc{M}_U,
\] 
so the (non-deformed) $!$-standard and $+$-standard $\mc{D}_Y$-modules $j_! \mc{M}_U$ and $j_+ \mc{M}_U$ appear as sub and quotient modules of the maximal extension $\Xi_f \mc{M}_U$ \cite[Lemma 4.2.1]{BBJantzen}. In this way, we obtain filtrations of the $\mc{D}_Y$-modules $j_! \mc{M}_U$ and $j_+ \mc{M}_U$ from the monodromy filtration of $\Xi_f \mc{M}_U$. These are the {\em geometric Jantzen filtrations}. 

Note that analogously to the algebraic Jantzen filtration, the geometric Jantzen filtration  depends on a choice of deformation parameter, given by the regular function $f:Y \rightarrow \A^1$. Moreover, the explicit realisation \eqref{eq: geometric Jantzen filtration as powers of s} in terms of powers of $s$ means that like the algebraic Jantzen filtration, the geometric Jantzen filtration is explicitly computable.

\subsubsection{Geometric Jantzen filtrations on Harish-Chandra sheaves}
The $\mc{D}$-modules corresponding to Verma modules and dual Verma modules under Beilinson--Bernstein localisation can be made to fit into the framework of \S\ref{sec: geometric Jantzen filtrations}, and thus acquire geometric Jantzen filtrations. Such $\mc{D}$-modules manifest as {\em Harish-Chandra sheaves}, which are a class of $\mc{D}$-modules equivariant with respect to a certain group action. We explain this connection below.

Let $G$ be the simply connected semisimple Lie group associated to $\mf{g}$, $B\subset G$ the Borel subgroup corresponding to $\mf{b}$, and $N \subset B$ its unipotent radical. Set $H:=B/N$ to be the abstract maximal torus of $G$ \cite[Lemma 6.1.1]{chrissginzburg}, and identify $\mf{h}$ with $\Lie H$. Let $\widetilde{X}:=G/N$ be the base affine space and $X:=G/B$ the flag variety. The projection $\pi: \widetilde{X} \rightarrow X$ is a principal $G$-equivariant $H$-bundle with respect to the right action of $H$ on $\widetilde{X}$ by right multiplication. 

\begin{remark}
\label{rem: H-monodromic D-modules} 
($H$-monodromic $\mc{D}_{\widetilde{X}}$-modules) In \cite{BBJantzen}, Beilinson-Bernstein work with $H$-monodromic $\mc{D}$-modules on base affine space $\widetilde{X}$ instead of modules over sheaves of twisted differential operators (TDOs) on the flag variety $X$, as they do in \cite{BB81}. Working over $\widetilde{X}$ has several advantages: it allows one to study entire families of representations at once (see Figures \ref{fig: dual Vermas} and \ref{fig: Vermas} for an illustration of this phenomenon), and it allows one to study $\mf{g}$-modules with generalised infinitesimal character\footnote{A $U(\mf{g})$-module $M$ has {\em generalised infinitesimal character} $\chi: Z(\mf{g})\rightarrow \C$ if for all $z \in Z(\mf{g})$ and $m \in M$, $(z - \chi(z))^k m = 0$ for some $k \in \Z_{> 0}$. }. In contrast, modules over TDOs can only be used to study $\mf{g}$-modules with strict infinitesimal character. There is a precise relationship between $H$-monodromic $\mc{D}_{\widetilde{X}}$-modules and modules over TDOs, see Remark \ref{rem: relationship between D tilde and TDOs}. 
\end{remark}

For an $N$-orbit\footnote{Note that this construction works for many Harish-Chandra pairs $(\mf{g},K)$, not just the pair $(\mf{g},N)$. In \cite[\S3.4]{BBJantzen}, the specific conditions on $K$ necessary for such a construction to hold are discussed. In particular, these constructions can be applied to any symmetric pair \cite[Lemma 3.5.2]{BBJantzen}, so they can be used in the study of admissible representations of real reductive groups.}(i.e. a Bruhat cell) $Q$ in $X$, denote by $\widetilde{Q} = \pi^{-1}(Q)$ the corresponding union of $N$-orbits in $\widetilde{X}$. A choice of dominant regular integral weight $\gamma \in \mf{h}^*$ (the ``deformation direction'') determines a regular function $ f_\gamma: \overline{\widetilde{Q}}\rightarrow \mathbb{A}^1$ on the closure of $\widetilde{Q}$ such that $f_\gamma^{-1}(\mathbb{A}^1 - \{0\}) = \widetilde{Q}$ \cite[Lemma 3.5.1]{BBJantzen}. This function extends to a regular function on $\widetilde{X}$, which, by the process outlined in \S\ref{sec: geometric Jantzen filtrations}, determines a maximal extension functor $\Xi_{f_\gamma}:\mc{M}_{\mathrm{hol}}(\mc{D}_U) \rightarrow \mc{M}_{\mathrm{hol}}(\mc{D}_{\widetilde{X}})$. Here $U$ is the preimage in $\widetilde{X}$ of $\A^1 - \{0\}$ under the extension of $f_\gamma$. Restricting $\Xi_{f_\gamma}$ to the category of holonomic $\mc{D}_U$-modules supported on $\widetilde{Q}$ results in a functor
\[
\Xi_{f_\gamma}: \mc{M}_{\mathrm{hol}}(\mc{D}_{\widetilde{Q}}) \rightarrow \mc{M}_{\mathrm{hol}} (\mc{D}_{\overline{\widetilde{Q}}}).
\]

 Let $\mc{O}_{\widetilde{Q}}$ be the structure sheaf on $\widetilde{Q}$ and $j_{\widetilde{Q}}: \widetilde{Q} \hookrightarrow \overline{\widetilde{Q}}$ the inclusion of $\widetilde{Q}$ into its closure. Via the construction in \S\ref{sec: geometric Jantzen filtrations}, the modules $j_{\widetilde{Q}!}\mc{O}_{\widetilde{Q}}$ and $j_{\widetilde{Q}+}\mc{O}_{\widetilde{Q}}$ acquire from $\Xi_{f_\gamma}\mc{O}_{\widetilde{Q}}$ geometric Jantzen filtrations. Because $\overline{\widetilde{Q}}$ is closed in $\widetilde{X}$, a theorem of Kashiwara \cite[Theorem 12.6]{D-modulesnotes} allows one to lift these filtrations to filtrations of the standard $N$-equivariant $\mc{D}_{\widetilde{X}}$-modules $i_{\widetilde{Q}!} \mc{O}_{\widetilde{Q}}$ and $i_{\widetilde{Q}+} \mc{O}_{\widetilde{Q}}$, for $i_{\widetilde{Q}}: \widetilde{Q} \hookrightarrow \widetilde{X}$ the inclusion. 

 There is a natural map  
 \[
 \mc{U}(\mf{g}) \rightarrow \Gamma(\widetilde{X}, \mc{D}_{\widetilde{X}}),
 \]
 obtained by differentiating the $G$-action on $\widetilde{X}$ which endows global sections of $\mc{D}_{\widetilde{X}}$-modules with the structure of $\mc{U}(\mf{g})$-modules. In \S\ref{sec: the map}, we explicitly compute this map for $\mf{sl}_2(\C)$. As $\mc{U}(\mf{g})$-modules, the global sections of $i_{\widetilde{Q}!} \mc{O}_{\widetilde{Q}}$ and $i_{\widetilde{Q}+} \mc{O}_{\widetilde{Q}}$ are direct sums of all integral Verma modules and dual Verma modules, respectively. In \S\ref{sec: vermas and dual vermas}, we illustrate this structure in our example. 

 \begin{remark} 
 \label{rem: comparison to Rom21}
 (Comparison to \cite{Rom21})
     It is interesting to contrast the computations of the current paper to those in Romanov's previous paper \cite{Rom21}, whose goal was to illustrate four families of $\mc{D}$-modules corresponding to well-known families of representations (finite-dimensional, Verma/dual Verma, principal series, and Whittaker). Our approach in the current paper is to study all integral Verma/dual Verma modules simultaneously by working over base affine space, as explained above. In contrast, \cite[\S6]{Rom21} analyses Verma/dual Verma modules one at a time using modules over varying TDOs on the flag variety. (Compare Figures \ref{fig: dual Vermas} and \ref{fig: Vermas} below to Figures 2 and 3 in \cite{Rom21}.) Our techniques in this paper are not specific to Verma modules: by working over base affine space, we could recover each family of examples in \cite{Rom21} using a single $H$-monodromic $\mc{D}$-module. 
     
     Our current approach is not merely stylistic --- it is necessary for our goal. Because the deformed Verma modules arising in the construction of the Jantzen filtration do not have a strict infinitesimal character as Verma modules do, they cannot be studied as modules over TDOs on the flag variety. However, deformed Verma modules can be approximated by $\mf{g}$-modules with generalised infinitesimal character (see \S\ref{sec: step 1}, and, in particular, \eqref{eq: generalised casimir action} and \eqref{eq: Casimir action on deformed dual vermas}), so a $\mc{D}$-module approach to their study must necessarily work over $\widetilde{X}$ instead of $X$, see Remark \ref{rem: H-monodromic D-modules}.
 \end{remark}

\subsubsection{Relationship between monodromy and weight filtrations}
\label{ref: relationship between monodromy and weight filtrations}

The geometric Jantzen filtration of $i_{\widetilde{Q}!}\mc{O}_{\widetilde{Q}}$ constructed in the previous section is computable via \eqref{eq: geometric Jantzen filtration as powers of s}, but it is not clear that it should satisfy the properties of Jantzen's conjectures. The key idea of Beilinson--Bernstein's proof is to relate the monodromy filtration on $\Xi_{f_\gamma} \mc{O}_{\widetilde{Q}}$ to the {\em weight filtration} on the corresponding perverse sheaf under the Riemann--Hilbert correspondence, which has strong functoriality and semisimplicity properties. 

Weight filtrations on objects in derived categories of constructible $\mathbb{Q}_\ell$-sheaves\footnote{Beilinson--Bernstein's results could also be formulated in the more modern language of Saito's mixed Hodge modules \cite{Saito1, Saito2}, but because the initial draft of their paper was written in 1986 before Saito's work was published, they instead used the technology of mixed $\ell$-adic sheaves \cite{Del}.} are a deep generalisation of filtrations on cohomology rings of algebraic varieties. Explicitly constructing weight filtrations is extremely difficult outside of the most basic examples, but they can be shown to exist for complexes built from simple examples via sheaf functors. In particular, the perverse sheaf corresponding to the maximal extension $\Xi_{f_\gamma} \mc{O}_{\widetilde{Q}}$ admits a `mixed structure', and hence a weight filtration, as it is the quotient of a push-forward of a $\mc{D}$-module of `geometric origin'. 

Beilinson--Bernstein's strategy was to utilize a theorem of Gabber \cite[Theorem 5.1.2]{BBJantzen}, which establishes that on a perverse sheaf obtained by a nearby cycles functor (of which the maximal extension functor is a special instance), the monodromy filtration agrees with the weight filtration. Passing Gabber's theorem to $\mc{D}$-modules via the Riemann--Hilbert correspondence lets them conclude that the geometric Jantzen filtration on $i_{\widetilde{Q}!}\mc{O}_{\widetilde{Q}}$ agrees with the weight filtration.  

Weight filtrations have two important properties: (1) they are functorial with respect to morphisms of mixed perverse sheaves, and (2) the associated graded object is semisimple. These properties are exactly what is needed to prove Jantzen's conjectures: the functoriality implies the strictness of the Jantzen filtration with respect to embeddings of Verma modules, and the semisimplicity of the associated graded implies (with some additional pointwise purity arguments) the agreement of the Jantzen filtration with the socle filtration. 

The power of Beilinson--Bernstein's proof comes from the connection between two very different filtrations --- the Jantzen filtration, which is explicitly computable, but has no obvious structure, and the weight filtration, which is very difficult to compute, but satisfies remarkable properties. 

\subsection{Relationship between algebraic and geometric Jantzen filtrations}
\label{sec: relationship between algebraic and geometric Jantzen filtrations}
Beilinson--Bernstein's proof of Jantzen's conjectures relies on the fact that the geometric and algebraic Jantzen filtrations align under the global sections functor. Though both constructions involve similar ingredients, such as deformations and relationships between standard and costandard objects, it is not immediately obvious from the definitions that they should yield the same filtration on Verma modules. This crucial relationship is given minimal justification in \cite{BBJantzen}.

Because of the critical nature of this relationship, we dedicate the final section of our paper \S\ref{sec: relation to the algebraic Jantzen filtration} to explicitly describing the relationship between the two filtrations for $\mf{sl}_2(\C)$, and illustrating it for a fixed infinitesimal character in Figure \ref{fig: comparing filtrations}. Our arguments easily generalise to any Lie algebra.

\subsection{Structure of the paper}
\label{sec: structure of the paper}
\hspace{3mm}
The remainder of the paper is dedicated to the computation of the geometric Jantzen filtration for the Lie algebra $\mf{sl}_2(\C)$. The computation is structured as follows. 

\S2.1: We establish an algebra homomorphism from the extended universal enveloping algebra to global differential operators on base affine space. This algebra homomorphism is what allows us to view the global sections of $\mc{D}_{\widetilde{X}}$-modules as modules over the (extended) universal enveloping algebra. 

\S2.2: We give some background on $H$-monodromic $\mc{D}_X$-modules, and explain their relationship to modules over twisted sheaves of differential operators. 

\S2.3: We introduce the $\mc{D}_{\widetilde{X}}$-modules whose global sections contain Verma modules and dual Verma modules --- these are the $\mc{D}_{\widetilde{X}}$-modules which we will endow with geometric Jantzen filtrations. We illustrate the $\mf{sl}_2(\C)$-module structure on their global sections in Figures \ref{fig: dual Vermas} and \ref{fig: Vermas}. 

\S2.4: We introduce the maximal extension functor, which gives the deformation necessary for the Jantzen filtration. We compute the maximal extension of the structure sheaf on an open subset of $\widetilde{X}$, and illustrate in Figures \ref{fig: deformed dual verma} and \ref{fig: deformed verma} how deformed Verma modules and deformed dual Verma modules arise geometrically. We illustrate in Figures \ref{fig: layers of maximal extension} and \ref{fig: maximal extension} the global sections of the maximal extension, identifying them with the big projective in category $\mc{O}$. 

\S2.5: We define the geometric Jantzen filtration using monodromy filtrations. We compute the monodromy filtration of the maximal extension, and illustrate its global sections in Figure \ref{fig: global sections of monodromy}. This specialises to the geometric Jantzen filtration on certain sub- and quotient sheaves. 

\S2.6: We introduce the algebraic Jantzen filtration on a Verma module in \S\ref{sec: the algebraic jantzen filtration}, then explain why the global sections of the geometric Jantzen filtration align with the algebraic Jantzen filtration in \S \ref{sec: relationship between algebraic and geometric Jantzen filtrations}. Figure \ref{fig: comparing filtrations} illustrates this relationship in our example.

\subsection{Acknowledgements}
This paper arose from computations in the first author's honours thesis at the University of New South Wales. We would like to thank the referee for helpful suggestions which greatly improved the readability of the paper. The first author would like to thank Daniel Chan, who influenced his interests in this topic, and broadened his understanding of the various geometric techniques used in this paper. The second author would like to thank Jens Eberhardt, Adam Brown, and Geordie Williamson for many hours of conversations about Jantzen filtrations which contributed significantly to her understanding. 

\section{Example}
\label{sec: example}

Now we proceed with our example. For the remainder of this paper, set $G = \SL_2(\C)$, and fix subgroups
\[
 B= \Set{ \bp a & b \\ 0 & a^{-1} \ep | \,a\in\C^*,\,b\in\C }, \quad 
 N = \Set{ \bp 1 & b \\ 0 & 1 \ep | b \in \C  }
\]
and 
\[
H = \Set{ \bp a & 0 \\ 0 & a^{-1} \ep | a \in \C^* }. 
\]
Let $\g$, $\mf{b}$, $\mf{n}$, and $\h$ be the corresponding Lie algebras, and $\bar{\n}$ the opposite nilpotent subalgebra to $\mf{n}$. Denote by 
\begin{equation}
\label{eq: e, f, h}
e = \bp 0 & 1 \\ 0 & 0 \ep, 
\quad f = \bp 0 & 0 \\ 1 & 0 \ep, \quad \text{and} \quad h = \bp 1 & 0 \\ 0 & -1 \ep 
\end{equation}
 the standard basis elements of $\mf{g}$, so $\mf{n} = \C e$, $\mf{h} = \C h$ and $\bar{\n} = \C f$. Denote by $\mc{Z}(\mf{g})$ the center of the universal enveloping algebra $\mc{U}(\mf{g})$. The algebra $\mathcal{Z}(\mf{g})$ is generated by the Casimir element 
 \begin{equation}
    \label{eq: casimir}
    \Omega = h^2 + 2ef + 2fe.
\end{equation}
Let 
 \begin{equation}
     \label{eq: HC projection}
     \gamma_{\mathrm{HC}}: \mc{U}(\g) \rightarrow \mc{U}(\h)
 \end{equation}
 be the projection onto the first coordinate of the direct sum decomposition
 \[
\mc{U}(\mf{g}) = \mc{U}(\mf{h}) \oplus (\bar{\mf{n}} \mc{U}(\mf{g}) + \mc{U}(\mf{g})\mf{n}).
 \]
The restriction of $\gamma_{\mathrm{HC}}$ to $\mathcal{Z}(\g)$ is an algebra homomorphism. 

Set $X=G/B$ and $\widetilde{X} = G/N$. Then $X$ is the flag variety of $\mf{g}$, and we refer to $\widetilde{X}$ as {\em base affine space}. We identify  $X$ with the complex projective line $\C\mathbb{P}^1$ via 
\begin{equation}
\label{eq: flag variety is P1}
\begin{pmatrix}x_1&*\\x_2&*\end{pmatrix}B\mapsto(x_1:x_2),
\end{equation}
and $\widetilde{X}$ with $\C^2 \backslash \{(0,0)\}$ via 
\begin{equation}
\label{eq: base affine space is C2}
\begin{pmatrix}x_1&*\\x_2&*\end{pmatrix}N\mapsto(x_1,x_2).
\end{equation}

There are left actions of $G$ on $X$ and $\widetilde{X}$ by left multiplication. Under the identifications \eqref{eq: flag variety is P1} and \eqref{eq: base affine space is C2}, these actions are given by
\begin{equation}
\label{eq: G action on flag variety}
\begin{pmatrix}a&b\\c&d\end{pmatrix}\cdot(x_1:x_2)=(ax_1+bx_2:cx_1+dx_2)
\end{equation}
and 
\begin{equation}
\label{eq: G action on base affine space}
\begin{pmatrix}a&b\\c&d\end{pmatrix}\cdot(x_1,x_2)=(ax_1+bx_2,cx_1+dx_2).
\end{equation}

Because $H$ normalizes $N$, there is also a right action of $H$ on $G/N$ by right multiplication. Under the identification \eqref{eq: base affine space is C2}, this action is given by 
\begin{equation}
\label{eq: H action on base affine space}
(x_1,x_2)\cdot\begin{pmatrix}a&0\\0&a^{-1}\end{pmatrix}=(ax_1,ax_2).
\end{equation}
The natural $G$-equivariant quotient map 
\begin{equation}
    \label{eq: pi}
    \pi: \widetilde{X} \rightarrow X
\end{equation}
is an $H$-torsor over $X$. In the language of \cite[\S 2.5]{BBJantzen}, this provides an ``$H$-monodromic structure'' on $X$. 

For an algebraic variety $Y$, we denote by $\mc{O}_Y$ the structure sheaf on $Y$, and by $\mc{O}(Y) = \Gamma(Y, \mc{O}_Y)$ the algebra of global regular functions. We denote by $\mc{D}_Y$ the sheaf of differential operators on $Y$, and $\mc{D}(Y) = \Gamma(Y, \mc{D}_Y)$ the global differential operators. 

Base affine space $\widetilde{X}$ is a quasi-affine variety, with affine closure $\overline{\widetilde{X}} = \A^2$. Throughout this text, we will make use the following facts about quasi-affine varieties. Let $Y$ be an irreducible quasi-affine variety, openly embedded in an affine variety $\overline{Y}$. 
\begin{itemize}
    \item If $Y$ is normal with $\mathrm{codim}_{\overline{Y}}(\overline{Y} \backslash Y) \geq 2$, then $\mc{O}(Y) = \mc{O}(\overline{Y})$ and $\mc{D}(Y) =\mc{D}(\overline{Y})$ \cite[\S2]{LS06}. (In particular, for $Y = \widetilde{X}$, this implies that global differential operators are nothing more than the Weyl algebra in 2 variables\footnote{Outside of $\mf{sl}_2(\C)$, the situation is less straightforward. For a Lie algebra $\mf{g} \neq \mf{sl}_2(\C)^m$, the affine closure of the corresponding base affine space is singular, and the ring of global differential operators can be quite complicated, see, for example, \cite{LS06}.}.) 
    \item Because the variety $\overline{Y}$ is affine, it is also $D$-affine, meaning that the global sections functor induces an equivalence of categories between the category of quasi-coherent $\mc{D}_{\overline{Y}}$-modules and the category of modules over $\mc{D}(\overline{Y})$. 
    \item Since the inclusion $i: Y \rightarrow \overline{Y}$ is an open immersion, the restriction functor $i^+$ on the corresponding categories of $\mc{D}$-modules is exact, and commutes with pushforwards from open affine subvarieties \cite[Remark 3.1]{D-modulesnotes}. 
\end{itemize}
The facts listed above allow us to move freely between $\mc{D}_{\widetilde{X}}$-modules and $\mc{D}(\A^2)$-modules. We will do this periodically in computations. 

\subsection{The map $\mc{U}(\g) \otimes_{\mc{Z}(\g)} \mc{U}(\h) \rightarrow \Gamma(\widetilde{X}, \mc{D}_{\widetilde{X}})$}
\label{sec: the map}

 Our strategy for gaining intuition about the $\mc{D}_{\widetilde{X}}$-modules arising in the construction of the Jantzen filtration is to illustrate the $\mf{g}$-module structure on their global sections. This will give us an algebraic snapshot as to what is happening at each step in the construction sketched in Section \ref{sec: The geometric Jantzen filtration}. The first step is to differentiate the actions \eqref{eq: G action on base affine space} and \eqref{eq: H action on base affine space} to obtain a map $\mc{U}(\g) \otimes_{\mc{Z}(\g)} \mc{U}(\h) \rightarrow \Gamma(\widetilde{X}, \mc{D}_{\widetilde{X}})$. This map provides the $\mf{g}$-module structure on the global sections of $\mc{D}_{\widetilde{X}}$-modules. We dedicate this section to the computation of this map. 

By differentiating the left action of $G$ in \eqref{eq: G action on base affine space}, we obtain an algebra homomorphism 
\begin{equation}
    \label{eq: the map U(g)--> D}
    L:\mc{U}(\mf{g}) \rightarrow \Gamma(\widetilde{X}, \mc{D}_{\widetilde{X}}),\hspace{2mm} g \mapsto L_g
\end{equation}
given by the formula 
\begin{equation}
    \label{eq: differentiating left G-action}
    L_g f(x) = \frac{d}{dt} \bigg|_{t=0} f (\mathrm{exp}(tg)^{-1}x)
\end{equation}
for $g \in G$, $f \in \Gamma(\widetilde{X}, \mc{O}_{
\widetilde{X}})$, $x \in \widetilde{X}$. Computing the image of the basis \eqref{eq: e, f, h} under the homomorphism \eqref{eq: the map U(g)--> D} is straighforward. For example, the image of $e$ is given by the following computation using \eqref{eq: G action on base affine space}. 
\begin{align*}
    e \cdot f(x_1, x_2) &= \frac{d}{dt} \bigg|_{t=0} f \left( \bp 1 & -t \\ 0 & 1 \ep \cdot (x_1, x_2 ) \right) \\
    &= \frac{d}{dt} \bigg|_{t=0} f(x_1 - t x_2, x_2) \\
    &= - x_2 \partial_1 f(x_1, x_2).
\end{align*}
Similar computations determine the image of $f$ and $h$: 
\begin{equation}
    \label{eq: differential operators corresponding to e, f, h} 
    L_e = -x_2 \partial_1, \quad L_f = -x_1 \partial_2, \quad L_h = -x_1 \partial_1 + x_2 \partial_2.
\end{equation}
It is also useful to compute the image of the Casimir element \eqref{eq: casimir} under the homomorphism $L$:
\begin{equation}
    \label{eq: image of casimir under L}
    L_\Omega = x_1^2 \partial_1^2 + 3 x_1 \partial_1 + 3 x_2 \partial_2 + x_2^2 \partial_2^2 + 2x_1 x_2 \partial_1 \partial_2.
\end{equation}

Similarly, the right action of $H$ determines an algebra homomorphism 
\begin{equation}
    \label{eq: the map U(h) --> D}
    R: \mc{U}(\mf{h}) \rightarrow \Gamma(\widetilde{X}, \mc{D}_{\widetilde{X}}), \quad g \mapsto R_g.
\end{equation}
Under this homomorphism, $h$ is sent to the Euler operator 
\begin{equation}
    \label{eq: right action of h}
    R_h = x_1 \partial_1 + x_2 \partial_2.
\end{equation}

Combining the homomorphisms $L$ \eqref{eq: the map U(g)--> D} and $R$ \eqref{eq: the map U(h) --> D}, we obtain an algebra homomorphism
\begin{equation}
    \label{eq: the map U(g) tensor U(h) --> D}
    \mc{U}(\mf{g}) \otimes_\C \mc{U}(\mf{h}) \rightarrow \Gamma(\widetilde{X},\mc{D}_{\widetilde{X}}); \quad g \otimes g' \mapsto L_g R_{g'}
\end{equation}
\begin{lemma}
\label{lem: factors through the quotient tilde U}
    The homomorphism \eqref{eq: the map U(g) tensor U(h) --> D} factors through the quotient 
    \begin{equation}
        \label{eq: U tilde}
        \widetilde{\mc{U}}:= \mc{U}(\mf{g}) \otimes _{\mc{Z}(\mf{g})} \mc{U}(\mf{h}),
    \end{equation}
    where $\mc{Z}(\mf{g})$ acts on $\mc{U}(\mf{h})$ via the Harish-Chandra projection $\gamma_\mathrm{HC}$ \eqref{eq: HC projection}. 
\end{lemma}
\begin{proof}
    Direct computation shows that the image of $\Omega \otimes 1$ and $1 \otimes \gamma_\mathrm{HC}(\Omega)$ agree. Indeed, 
\begin{align*}
1 \otimes \gamma_{\mathrm{HC}}(\Omega) = 1 \otimes (h^2 + 2h) \mapsto &R_h^2 + 2 R_h \\
&=(x_1 \partial_1 + x_2 \partial_2)^2 + 2(x_1 \partial_1 + x_2 \partial_2) \\
&=x_1^2 \partial _1^2 + 3 x_1 \partial_1 + 2 x_1 x_2 \partial_1 \partial_2 + x_2^2 \partial_x^2 + 3 x_2 \partial_x \\
&= L_\Omega.
    \end{align*}
\end{proof}

We refer to the algebra $\widetilde{\mc{U}}$ as the {\em extended universal enveloping algebra}. By Lemma \ref{lem: factors through the quotient tilde U}, we have an algebra homomorphism 
\begin{equation}
    \label{eq: the map tilde U --> D}
    \alpha: \widetilde{\mc{U}} \rightarrow \Gamma(\widetilde{X}, \mc{D}_{\widetilde{X}}); \quad g \otimes g' \mapsto L_g R_{g'}.
\end{equation}
Global sections of $\mc{D}_{\widetilde{X}}$-modules have the structure of $\widetilde{\mc{U}}$-modules via $\alpha.$   

\subsection{Monodromic $\mc{D}_{X}$-modules} 
\label{sec: monodromic D-modules}

The $\mc{D}$-modules which play a role in our story have an additional structure: they are ``$H$-monodromic''. It is necessary for our purposes to work with $H$-monodromic $\mc{D}$-modules on base affine space instead of $\mc{D}$-modules on the flag variety. This is due to the fact that the $\mf{g}$-modules in the construction of the Jantzen filtration have generalized infinitesimal character, so they do not arise as global sections of modules over twisted sheaves of differential operators on the flag variety. 

The machinery of $H$-monodromic $\mc{D}$-modules is rather technical, and the details of the construction are not strictly necessary for our computation of the Jantzen filtration below. However, we thought that it would be useful to describe this construction in a specific example to illustrate that the equivalences established in \cite[\S 2.5]{BBJantzen} are quite clear for $\mf{sl}_2(\C)$. In this section, we describe the construction of $H$-monodromic $\mc{D}$-modules for $\mf{sl}_2(\mathbb{\C})$ and explain how it relates to representations of Lie algebras. More details on the general construction can be found in \cite{BBJantzen, BG99}.

\begin{definition}
    \label{def: H-monodromic D-module} 
    An {\em $H$-monodromic $\mc{D}_X$-module} is a weakly $H$-equivariant\footnote{A {\em weakly $H$-equivariant $\mc{D}_{\widetilde{X}}$-module} is an $H$-equivariant sheaf $\mc{V}$ equipped with a $\mc{D}_{\widetilde{X}}$-module structure so that the isomorphism $\mathrm{act}^*\mc{V} \rightarrow p^* \mc{V}$ given by the equivariant sheaf structure on $\mc{V}$ is a morphism of $\mc{D}_{\widetilde{X}} \boxtimes \mc{O}_H$-modules. Here $\mathrm{act}:\widetilde{X} \times H \rightarrow \widetilde{X}$ is the action map and $p: \widetilde{X} \times H \rightarrow \widetilde{X}$ is the projection map. For a reference on weakly equivariant $\mc{D}$-modules, see \cite[\S 4]{MP}.} $\mc{D}_{\widetilde{X}}$-module.
\end{definition}

There is an equivalent characterization of $H$-monodromic $\mc{D}_X$-modules in terms of $H$-invariant differential operators which is established in \cite[\S2.5.2]{BBJantzen}. This perspective makes the structures of our examples more transparent, so we will take this approach to monodromicity. Below we describe the construction for $\mf{g}=\mf{sl}_2(\C)$.   

The right $H$-action in \eqref{eq: H action on base affine space} induces a left $H$-action on $\mc{O}_{\widetilde{X}}$ and $\mc{D}_{\widetilde{X}}$. The $H$-action on $\mc{D}_{\widetilde{X}}$ satisfies the following relation: for $g \in H$, $\theta \in \mc{D}_{\widetilde{X}}$, and $f \in \mc{O}_{\widetilde{X}}$,
\begin{equation}
    \label{eq: H action on differential operators}
    (g \cdot \theta) (g \cdot f) = g \cdot (\theta(f)). 
\end{equation}

The $H$-action on $\mc{D}_{\widetilde{X}}$ induces an $H$-action on the sheaf $\pi_*(\mc{D}_{\widetilde{X}})$ by algebra automorphisms, where $\pi:\widetilde{X}\rightarrow X$ is the quotient map \eqref{eq: pi}. Here $\pi_*$ is the $\mc{O}$-module direct image. Denote the sheaf of $H$-invariant sections of $\pi_*\mc{D}_{\widetilde{X}}$ by 
\begin{equation}
    \label{eq: D tilde}
    \widetilde{\mc{D}}: = [\pi_*\mc{D}_{\widetilde{X}}]^H. 
\end{equation}
This is a sheaf of algebras on $X$. Explicitly, on an open set $U \subseteq X$, 
\begin{equation}
    \label{eq: explicit definition of D tilde}
    \widetilde{\mc{D}}(U) = \mc{D}_{\widetilde{X}}(\pi^{-1}(U))^H.
\end{equation}
Note that because $\pi$ is an $H$-torsor, $\pi^{-1}(U)$ is $H$-stable for any set $U$, so this construction is well-defined. 

Let $\mc{M}(\mc{D}_{\widetilde{X}}, H)_{\mathrm{weak}}$ be the category of weakly $H$-equivariant $\mc{D}_{\widetilde{X}}$-modules, and $\mc{M}(\widetilde{\mc{D}})$ be the category of $\widetilde{\mc{D}}$-modules. 
By \cite[\S 1.8.9, \S 2.5.2]{BBJantzen}, there is an equivalence of categories
\begin{equation}
\label{eq: equivalent notions of monodromic D-modules}
\mc{M}(\mc{D}_{\widetilde{X}}, H)_{\mathrm{weak}} \simeq \mc{M}(\widetilde{\mc{D}}).
\end{equation}
Hence we can study monodromic $\mc{D}_X$-modules by instead considering $\widetilde{\mc{D}}$-modules. For the remainder of the paper, we will take this to be our definition of monodromicity. 

\begin{definition}
    \label{def: monodromic take 2}
    An {\em $H$-monodromic $\mc{D}_X$-module} is a $\widetilde{\mc{D}}$-module, where $\widetilde{\mc{D}}$ is as in \eqref{eq: D tilde}. 
\end{definition}

\begin{remark}
    \label{rem: relationship between D tilde and TDOs}
    (Relationship to twisted differential operators) The sheaf $\widetilde{\mc{D}}$ is a sheaf of $S(\h)$-algebras. In our example, the $S(\h)$-action is given by multiplication by the operator $R_h$ \eqref{eq: right action of h}. In particular, we can consider $S(\h)$ as a subsheaf of $\widetilde{\mc{D}}$. In fact, it is the center \cite[\S 2.5]{BBJantzen}. For $\lambda \in \h^*$, denote by $\mf{m}_\lambda \subset S(\h)$ the corresponding maximal ideal. The sheaf $\mc{D}_\lambda:= \widetilde{\mc{D}}/ \mf{m}_\lambda \widetilde{\mc{D}}$ is a twisted sheaf of differential operators (TDO) on $X$. Hence $\widetilde{\mc{D}}$-modules on which $R_h$ acts by eigenvalue $\lambda$ can be naturally identified with modules over the TDO $\mc{D}_\lambda$. 
\end{remark}

Modules over $\widetilde{\mc{D}}$ are directly related to modules over the extended universal enveloping algebra \eqref{eq: U tilde} via the global sections functor. The relationship is as follows. Because the left $G$-action and right $H$-action commute, the differential operators $L_e, L_f,$ and $L_h$ in \eqref{eq: differential operators corresponding to e, f, h} are $H$-invariant\footnote{This can also be shown via direct computation using \eqref{eq: differential operators corresponding to e, f, h} and \eqref{eq: right action of h}.}. Hence the image of the homomorphism \eqref{eq: the map tilde U --> D} is contained in $H$-invariant differential operators: 
\[
\alpha(\widetilde{\mc{U}}) \subseteq \Gamma(\widetilde{X}, \mc{D}_{\widetilde{X}})^H.
\]
Composing $\alpha$ with $\Gamma(\pi_*)$, we obtain a map 
\begin{equation}
    \label{eq: U tilde is global sections of D tilde}
    \widetilde{\mc{U}} \rightarrow \Gamma(X, \widetilde{\mc{D}}). 
\end{equation}

\begin{theorem}
    \label{thm: global sections of D tilde is U tilde}
    \cite[Lemma 3.2.2]{BBJantzen} The map \eqref{eq: U tilde is global sections of D tilde} is an isomorphism.
\end{theorem}

\begin{proof}
    The theorem holds for general $\mf{g}$. We will prove the theorem for $\mf{g}=\mf{sl}_2(\C)$ by direct computation. 

    We start by describing the sheaves $\pi_*\mc{D}_{\widetilde{X}}$ and $\widetilde{\mc{D}}$ on $X=\C \mathbb{P} ^1$ by describing them on the open patches 
    \begin{equation}
        \label{eq: open cover of P1}
        U_1 := \C \mathbb{P}^1 \backslash \{(0:1) \}  \quad \text{and} \quad U_2:= \C \mathbb{P}^1 \backslash \{ (1:0)\}
    \end{equation} 
    and giving gluing conditions. Set 
    \begin{equation}
        V_1:= \pi^{-1}(U_1) = \C^2 \backslash V(x_1) \quad \text{and} \quad V_2:= \pi^{-1}(U_2) =\C^2 \backslash V(x_2),
    \end{equation}
    where $V(f(x_1, x_2))$ denotes the vanishing of the polynomial $f(x_1, x_2)$. By definition, we have 
    \begin{align}
        \pi_*\mc{D}_{\widetilde{X}}(U_1) &= \mc{D}_{\widetilde{X}}(V_1) = \mc{D}(\C^2)[x_1^{-1}], \\
        \pi_*\mc{D}_{\widetilde{X}}(U_2) &= \mc{D}_{\widetilde{X}}(V_2) = \mc{D}(\C^2)[x_2^{-1}].
    \end{align}
    with obvious gluing conditions. 
    
    Using \eqref{eq: H action on differential operators}, we conclude that the $H$-action on $\mc{D}_{\widetilde{X}}$ is given by the local formulas
    \begin{equation}
        \label{eq: H action on differential operators sl2}
        g \cdot x_i = gx_i \quad \text{and} \quad g \cdot \partial_i = g^{-1} \partial_i,
    \end{equation}
    where $g \in H$ is regarded as an element of $\C^\times$ under the identification 
    \[
    H \simeq \C^\times, \bp a & 0 \\ 0 & a^{-1} \ep \mapsto a.
    \]
    From this we obtain a local description of $\widetilde{\mc{D}}$, using \eqref{eq: explicit definition of D tilde}:
    \begin{align}
    \label{eq: local structure of D tilde for sl2}
\widetilde{\mc{D}}(U_1) &= \langle x_1^{-1}x_2, x_1 \partial_1, x_1 \partial_2, x_2 \partial_1, x_2 \partial_2 \rangle \subseteq \mc{D}_{\widetilde{X}}(V_1) \\ 
    \widetilde{\mc{D}}(U_2)&= \langle x_1 x_2^{-1}, x_1 \partial_1, x_1 \partial_2, x_2 \partial_1, x_2 \partial_2 \rangle \subseteq \mc{D}_{\widetilde{X}}(V_2). 
    \end{align}
Hence the global sections are given by 
\begin{equation}
    \label{eq: global sections of D tilde for sl2}
    \Gamma(X, \widetilde{\mc{D}}) \simeq \langle x_1 \partial_1, x_1 \partial_2, x_2 \partial_1, x_2 \partial_2 \rangle =\Gamma(\widetilde{X}, \mc{D}_{\widetilde{X}})^H \subseteq \Gamma(\widetilde{X}, \mc{D}_{\widetilde{X}}). 
\end{equation}

Now, it is clear that 
\begin{equation}
    \label{eq: image of U tilde for sl2}
L_e = -x_2 \partial_1, \quad L_f = -x_1 \partial_2, \quad L_h + R_h = 2x_2 \partial_2 \quad \text{and} \quad L_h-R_h = -2x_1 \partial_1
\end{equation}
so the operators $L_e, L_f, L_h$ and $R_h$ generate $\Gamma(\widetilde{X}, \mc{D}_{\widetilde{X}})^H$. Since $e\otimes 1, f \otimes 1, h \otimes 1$ and $1 \otimes h$ generate $\widetilde{\mc{U}}$, this shows that the map \eqref{eq: U tilde is global sections of D tilde} is surjective. Direct computations establish that
$$[L_e,L_f]=x_2\partial_1x_1\partial_2-x_1\partial_2x_2\partial_1=x_2\partial_2-x_1\partial_1=L_h,$$
$$[L_e,L_h]=x_2\partial_1(x_1\partial_1-x_2\partial_2)-(x_1\partial_1-x_2\partial_2)x_2\partial_1=2x_2\partial_1=-2L_e,$$
$$[L_f,L_h]=x_1\partial_2(x_1\partial_1-x_2\partial_2)-(x_1\partial_1-x_2\partial_2)x_1\partial_2=-2x_1\partial_2=2L_f,$$
$$[L_e,R_h]=[L_f,R_h]=[L_h,R_h]=0.$$
Combining these computations with the fact that $L_e,L_f,L_h$ and $R_h$ are linearly independent shows that the relations satisfied by $L_e, L_f, L_h,$ and $R_h$ are precisely those satisfied by $e\otimes1,f\otimes1,h\otimes1$ and $1\otimes h.$ Therefore, the map \eqref{eq: U tilde is global sections of D tilde} is also injective. 
\end{proof}

The relationships described in this section can be summarised with the following commuting diagrams. 
\begin{equation}
\label{eq: big diagram}
    \begin{tikzcd}
        \mc{M}_{coh}(\mc{D}_{\widetilde{X}},H)_{\mathrm{weak}} \arrow[rr, "\text{forget equiv.}"] \arrow[d, "\Gamma"'] & & \mc{M}_{coh}(\mc{D}_{\widetilde{X}})  \arrow[rr, "\pi_*"] & & \mc{M}_{coh}(\pi_* \mc{D}_{\widetilde{X}}) \arrow[d, "\text{restrict}"] \\
        \mc{M}_{f.g.}(\widetilde{\mc{U}}) & & & & \mc{M}_{coh}(\widetilde{\mc{D}}) \arrow[llll, "\Gamma"']
    \end{tikzcd}
\end{equation}
The composition of the top two arrows and the right-most arrow is the equivalence \eqref{eq: equivalent notions of monodromic D-modules} (See \cite[\S1.8.9, \S2.5.3]{BBJantzen} for more details.)

 \subsection{Verma modules and dual Verma modules}
\label{sec: vermas and dual vermas} Using the map \eqref{eq: the map tilde U --> D} constructed in Section \ref{sec: the map}, we can describe the $\widetilde{\mc{U}}$-module structure on various classes of $\mc{D}_{\widetilde{X}}$-modules. We will start by examining the $\mc{D}_{\widetilde{X}}$-modules $j_+\mc{O}_U$ and $j_!\mc{O}_U$, where $j:U \hookrightarrow \widetilde{X}$ is inclusion of the open union of $N$-orbits 
\begin{equation}
    \label{eq: open B orbit U}
    U := \C^2 \backslash V(x_2).
\end{equation}
Here the $+$ and $!$ indicate the $\mc{D}$-module push-forward functors, see \cite{D-modulesnotes}. These are the $\mc{D}_{\widetilde{X}}$-modules which will eventually be endowed with geometric Jantzen filtrations in Section \ref{sec: monodromy filtration and Jantzen filtration}. In this section, we describe the $\widetilde{\mc{U}}$-module structure on $\Gamma(\widetilde{X}, j_+\mc{O}_U)$ and $\Gamma(\widetilde{X}, j_!\mc{O}_U)$. 

Because $j$ is an open embedding, the $\mc{D}_{\widetilde{X}}$-module $j_+\mc{O}_U$ is just the sheaf $\mc{O}_U$ with $\mc{D}_{\widetilde{X}}$-module structure given by the restriction of $\mc{D}_U$ to $\mc{D}_{\widetilde{X}} \subseteq \mc{D}_U$. Hence the global sections of $j_+\mc{O}_U$ can be identified with the ring 
\begin{equation}
    \label{eq: global sections of J+}
    \Gamma(\widetilde{X}, j_+\mc{O}_U) = \C[x_1, x_2, x_2^{-1}].
\end{equation}

The operators $L_e, L_f, L_h$, and $R_h$ from \eqref{eq: differential operators corresponding to e, f, h} and \eqref{eq: right action of h} act on monomials $x_1^m x_2^n$ for $m \geq 0$ $n \in \Z$ by the  formulas
\begin{align}
    \label{eq: e action on monomials}
    L_e\cdot x_1^mx_2^n&=-mx_1^{m-1}x_2^{n+1}, \\
    \label{eq: f action on monomials}
L_f\cdot x_1^mx_2^n&=-nx_1^{m+1}x_2^{n-1}, \\
\label{eq: left h action on monomials}
L_h\cdot x_1^mx_2^n&=(n-m)x_1^mx_2^n,\\
\label{eq: right h action on monomials}
R_h\cdot x_1^mx_2^n&=(m+n)x_1^mx_2^n.
\end{align}

Using \eqref{eq: e action on monomials}-\eqref{eq: right h action on monomials}, we can illustrate the $\widetilde{\mc{U}}$-module structure on $\Gamma(\widetilde{X}, j_+\mc{O}_U)$ using nodes and colored arrows. We do this in Figure \ref{fig: dual Vermas}. The monomials $x_1^m x_2^n$ for $m \in \Z_{\geq 0}$ and $n \in \Z$ form a basis for $\Gamma(\widetilde{X}, j_+\mc{O}_U)$. The {\color{ForestGreen} green} arrows illustrate the action of the operator {\color{ForestGreen} $L_e$} on basis elements, the {\color{red} red} arrows the action of {\color{red} $L_f$}, and the {\color{blue} blue} arrows the action of {\color{blue} $L_h$}. If an operator acts by zero, no arrow is included. The {\color{Gray} $R_h$}-eigenspaces are highlighted in grey, with corresponding eigenvalues listed below.  

\begin{figure}[h]
\centering
\makebox[\textwidth][c]{\includegraphics[width=1.2\textwidth]{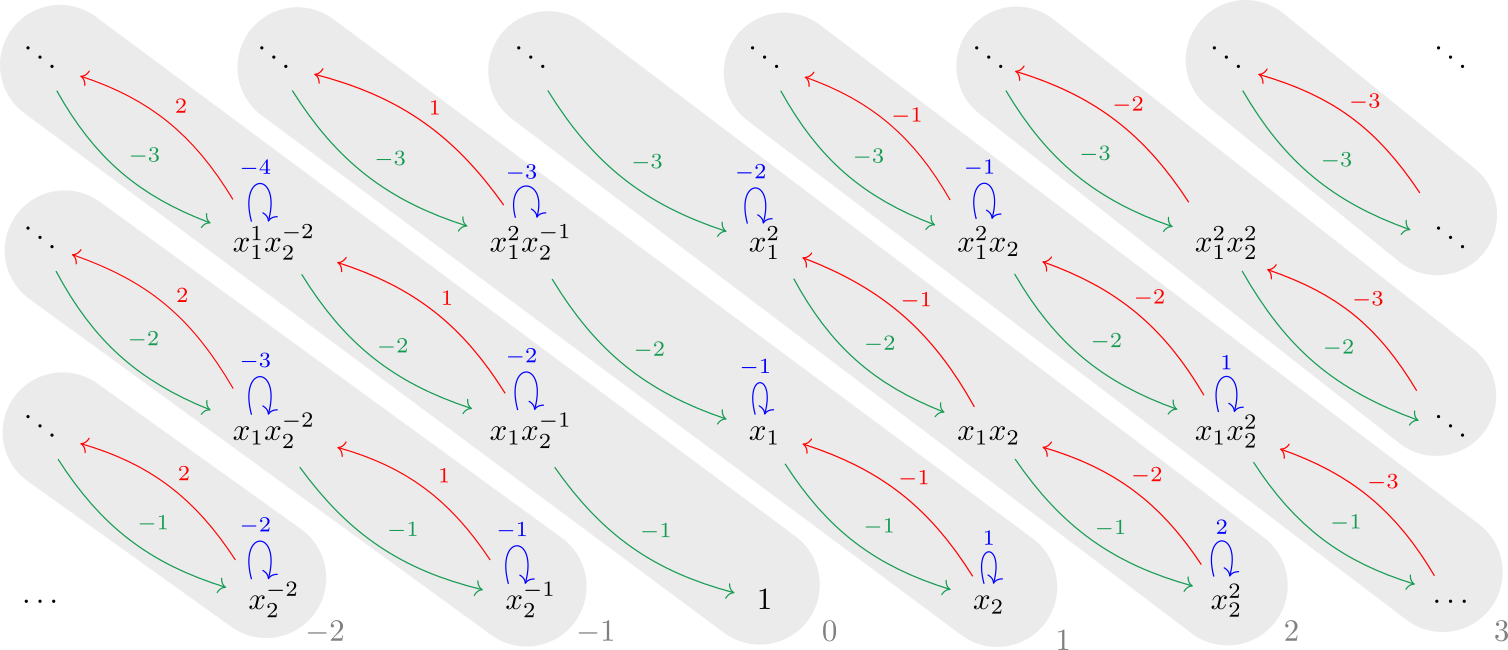}}
\caption{Dual Verma modules arise as global sections of $j_+\mc{O}_U$.}
\label{fig: dual Vermas}
\end{figure}

\begin{remark}
    \label{rem: observations about dual Vermas}
We make the following observations about the $\mc{D}_{\widetilde{X}}$-module  $j_+\mc{O}_U$ and its global sections. 
    \begin{enumerate}
        \item As a $\widetilde{\mc{U}}$-module, $\Gamma(\widetilde{X}, j_+\mc{O}_U)$ decomposes into a direct sum of submodules, each of which is an $R_h$-eigenspace corresponding to an integer eigenvalue:
        \[
\Gamma(\widetilde{X}, j_+ \mc{O}_U) = \bigoplus_{n \in \Z} \Gamma(\widetilde{X}, j_+ \mc{O}_U)_n
        \]
        In Figure \ref{fig: dual Vermas}, these eigenspaces are highlighted in grey. 
        \item As a $\mc{U}(\mf{g})$-module, the $R_h$-eigenspace $\Gamma(\widetilde{X}, j_+\mc{O}_U)_n$ of eigenvalue $n$ is isomorphic to the dual Verma module of highest weight $n$. In particular, it is irreducible if $n<0$, and it has a unique irreducible finite-dimensional submodule if $n \geq 0$.

        \item The sheaf $\pi_*j_+\mc{O}_U$ is a monodromic $\mc{D}_X$-module because it admits an action of $\widetilde{\mc{D}}$ (Definition \ref{def: monodromic take 2}). For each positive integer $n$, $\pi_*j_+\mc{O}_{U}$ has a  subsheaf $(\pi_*j_+\mc{O}_{U})_n$ on which $R_h$ acts locally by the eigenvalue $n$. These subsheaves are $\mc{D}_n$-modules, where $\mc{D}_n$ is the twisted sheaf of differential operators as defined in Remark \ref{rem: relationship between D tilde and TDOs}. These are exactly the $\mc{D}_n$-modules appearing in \cite[\S6, Fig. 4]{Rom21}. 
    \end{enumerate}
    \end{remark}

Next we will describe $\Gamma(\widetilde{X}, j_!\mc{O}_{\widetilde{X}})$. This is slightly more involved. By definition, 
\begin{equation}
\label{eq: j shriek}
     j_!= \mathbb{D}_{\widetilde{X}} \circ j_+ \circ \mathbb{D}_U,
\end{equation}
where $\mathbb{D}$ denotes the holonomic duality functor. Explicitly, for a smooth algebraic variety $Y$ and a holonomic $\mc{D}_Y$-module $\mc{V}$, 
\begin{equation}
\label{eq: duality functor}
    \mathbb{D}_Y(\mc{V}):= \Ext_{\mc{D}_Y}^{\dim Y}(\mc{V}, \mc{D}_X).
\end{equation}
This is a well-defined functor from the category of holonomic $\mc{D}_Y$-modules to itself \cite[Corollary 2.6.8]{HoTa}.

The first two steps of the composition \eqref{eq: j shriek} are straightforward to compute. The right $\mc{D}_U$-module $\mathbb{D}_U \mc{O}_U$ is just the sheaf $\mc{O}_U$,  viewed as a right $\mc{D}_U$-module via the natural right action. Then, since $j$ is an open immersion, $j_+ \mathbb{D}_U \mc{O}_U$ is the sheaf $\mc{O}_U$ with right $\mc{D}_{\widetilde{X}}$-module structure given by restriction to $\mc{D}_{\widetilde{X}} \subset \mc{D}_U$. 

To apply $\mathbb{D}_{\widetilde{X}}$ to $j_+ \mathbb{D}_U \mc{O}_U$, we must take a projective resolution of $j_+\mathbb{D}_U \mc{O}_U$. First, we make the identification
\begin{equation}
    \label{eq: identify j+ as a quotient}
    j_+ \mathbb{D}_U \mc{O}_U \simeq \langle \partial_1, \partial_2 \rangle \mc{D}_U \backslash \mc{D}_U.
\end{equation}
We take the following free (hence projective) resolution of $\langle \partial_1, \partial_2 \rangle \mc{D}_U \backslash \mc{D}_U$:
\begin{equation}
    \label{eq: projective resolution}
    0 \leftarrow \langle \partial_1, \partial_2 \rangle \mc{D}_U \backslash \mc{D}_U \xleftarrow{\epsilon} \mc{D}_{\widetilde{X}} \xleftarrow{d_0} \mc{D}_{\widetilde{X}} \oplus \mc{D}_{\widetilde{X}} \xleftarrow{d_1} \mc{D}_{\widetilde{X}} \xleftarrow{d_2} 0
\end{equation}
where the maps are defined by 
\begin{align*}
    \epsilon: 1 &\mapsto x_2^{-1}, \\
    d_0: (\theta_1, \theta_2) &\mapsto \partial_1 \theta_1 - x_2 \partial_2 \theta_2, \\
    d_1: 1 &\mapsto (x_2 \partial_2, \partial_1).
\end{align*}
Applying $\mathrm{Hom}_{\mc{D}_{\widetilde{X}}, r}( - , \mc{D}_{\widetilde{X}})$ to this complex, we obtain the complex 
\begin{equation}
    \label{eq: apply hom}
    0 \rightarrow \mathrm{Hom}_{\mc{D}_{\widetilde{X}},r}(\mc{D}_{\widetilde{X}}, \mc{D}_{\widetilde{X}}) \xrightarrow{d_0^*} \mathrm{Hom}_{\mc{D}_{\widetilde{X}},r}(\mc{D}_{\widetilde{X}}\oplus \mc{D}_{\widetilde{X}}, \mc{D}_{\widetilde{X}}) \xrightarrow{d_1^*} \mathrm{Hom}_{\mc{D}_{\widetilde{X}},r}(\mc{D}_{\widetilde{X}}, \mc{D}_{\widetilde{X}}) \xrightarrow{d_2^*} \ 0 
\end{equation}
where $d_i^*$ sends a morphism $f$ to $f \circ d_i$. 

Because the module $j_+\mathbb{D}_U \mc{O}_U$ is holonomic, the complex \eqref{eq: apply hom} only has nonzero cohomology in degree $2$. This can also be seen by direct computation. By identifying $\mathrm{Hom}_{\mc{D}_{\widetilde{X}},r}(\mc{D}_{\widetilde{X}}, \mc{D}_{\widetilde{X}}) \simeq \mc{D}_{\widetilde{X}}$ via $f \mapsto f(1)$ and $\mathrm{Hom}_{\mc{D}_{\widetilde{X}},r}(\mc{D}_{\widetilde{X}}\oplus \mc{D}_{\widetilde{X}}, \mc{D}_{\widetilde{X}}) \simeq \mc{D}_{\widetilde{X}} \oplus \mc{D}_{\widetilde{X}}$ via $ f \mapsto (f(1,0), f(0,1))$, we see that 
\[
\ker d_2^* \simeq \mc{D}_{\widetilde{X}} \text{ and } \im d_1^* \simeq \mc{D}_{\widetilde{X}} \langle  \partial_1, x_2 \partial_2 \rangle. 
\]
Hence,
\begin{equation}
    \label{eq: explicit description of j shriek}
    j_!\mc{O}_U \simeq \mc{D}_{\widetilde{X}} / \mc{D}_{\widetilde{X}} \langle \partial_1, x_2 \partial_2 \rangle.
\end{equation}

Now we can describe the global sections of $j_! \mc{O}_U$ and illustrate their $\widetilde{\mc{U}}$-module structure, as we did for $\mc{O}_{\widetilde{X}}$ and $j_+\mc{O}_U$. The monomials $x_1^m x_2^n$ and $x_1^m \partial_2^n$ for $m, n \geq 0$ form a basis for $\Gamma(\widetilde{X}, j_! \mc{O}_U)$. The action of $L_e, L_f, L_h$ and $R_h$ on $x^m_1 x_2^n$ for $m \geq 0$ and $n > 0$ is given by equations \eqref{eq: e action on monomials}-\eqref{eq: right h action on monomials}. The action of $L_e, L_f, L_h$ and $R_h$ on $x^m_1 \partial_2^n$ for $m \geq 0$ and $n > 0$ is given by 
\begin{align}
    \label{eq: e action on mixed monomials}
    L_e \cdot x_1^m \partial_2^n &= m(n-1) x_1 ^{m-1} \partial_2^{n-1}, \\
    \label{eq: f action on mixed monomials} 
    L_f \cdot x_1^m \partial_2^n &= -x_1^{m+1} \partial_2^{n+1}, \\
    \label{eq: h action on mixed monomials}
    L_h \cdot x_1^m \partial_2^n &= -(m+n) x_1^m \partial_2^n,\\
    \label{eq: right h action on mixed monomials}
    R_h \cdot x_1^m \partial_2^n &= (m-n) x_1^m \partial_2^n .
\end{align}
The action of $L_e$ on $x_1^m$ is given by \eqref{eq: e action on monomials}, the action of $L_f$ on $x_1^m$ is given by \eqref{eq: f action on mixed monomials}, and the actions of $L_h$ and $R_h$ on $x^m$ are given by either \eqref{eq: left h action on monomials}-\eqref{eq: right h action on monomials} or \eqref{eq: h action on mixed monomials}-\eqref{eq: right h action on mixed monomials}. 

We illustrate the $\widetilde{\mc{U}}$-module structure of $\Gamma(\widetilde{X}, j_! \mc{O}_U)$ in Figure \ref{fig: Vermas}.  The colors indicate the same operators as in the earlier example: {\color{ForestGreen} green} is {\color{ForestGreen} $L_e$}, {\color{red} red} is {\color{red} $L_f$}, {\color{blue} blue} is {\color{blue} $L_h$}, and {\color{Gray} $R_h$}-eigenspaces are highlighted in grey, with corresponding eigenvalues listed below.    

\begin{figure}[h]
\makebox[\textwidth][c]{\includegraphics[width=1.2\textwidth]{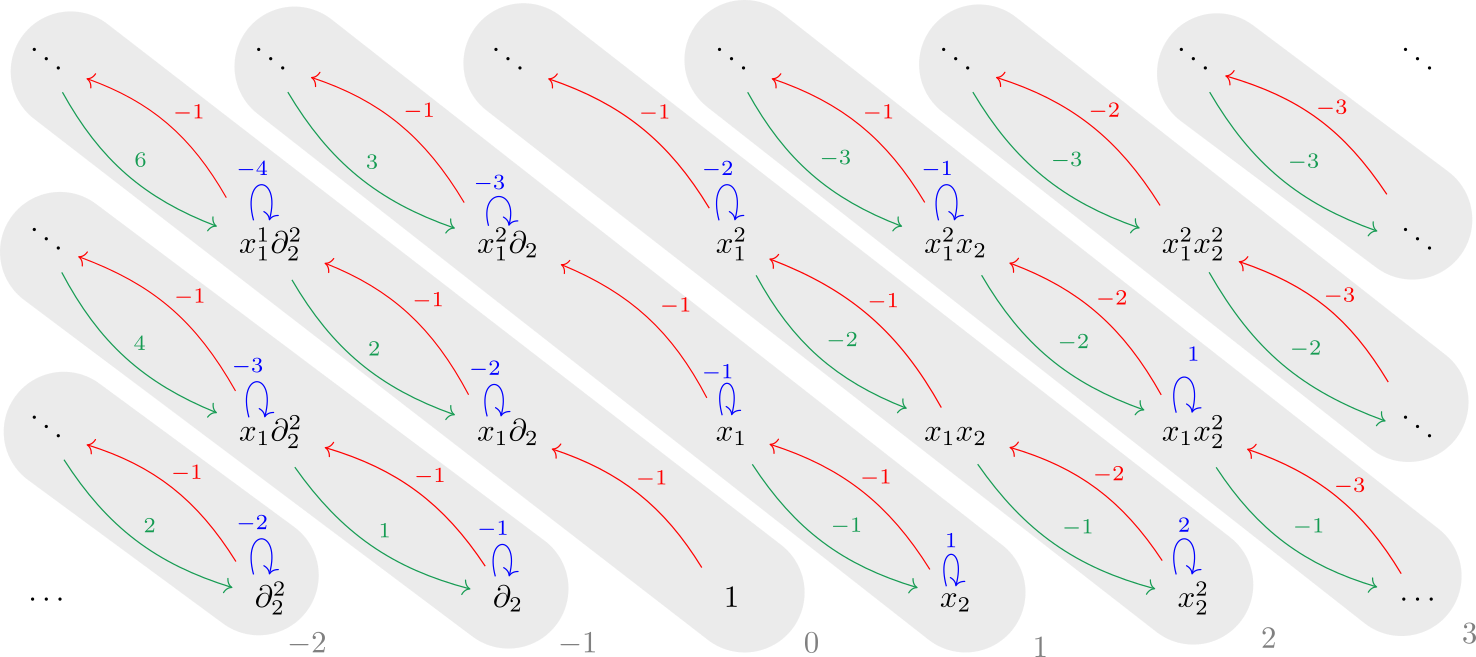}}
\caption{Verma modules arise as global sections of $j_!\mc{O}_U$.}
\label{fig: Vermas}
\centering
\end{figure}

\begin{remark}
    \label{rem: observations about Vermas} 
    We make the following observations about $\Gamma(\widetilde{X}, j_!\mc{O}_U)$. 
    \begin{enumerate}
        \item As a $\widetilde{\mc{U}}$-module, $\Gamma(\widetilde{X}, j_!\mc{O}_U)$ decomposes into a direct sum of submodules, each of which is an $R_h$-eigenspace corresponding to an integer eigenvalue. Again, these eigenspaces are highlighted in grey.
        \item As a $\mc{U}(\mf{g})$-module, the $R_h$-eigenspace of $\Gamma(\widetilde{X}, j_!\mc{O}_U)$ of eigenvalue $n$ is isomorphic to the Verma module of highest weight $n$. In particular, it is irreducible if $n<0$, and it has a unique irreducible finite-dimensional quotient if $n \geq 0$.
        \item The sheaf $\pi_*j_!\mc{O}_U$ is an $H$-monodromic $\mc{D}_X$-module. For each positive integer $n$, $\pi_*j_!\mc{O}_{U}$ has a  subsheaf $(\pi_*j_!\mc{O}_{U})_n$ on which $R_h$ acts locally by the eigenvalue $n$. These subsheaves are modules over the TDO $\mc{D}_n$ (Remark \ref{rem: relationship between D tilde and TDOs}). These are exactly the $\mc{D}_n$-modules appearing in \cite[\S6, Fig. 3]{Rom21}.
    \end{enumerate}
\end{remark}

\subsection{The maximal extension $\Xi_\rho\mc{O}_U$}
\label{sec: maximal extension}

To describe the geometric Jantzen filtrations on the $\mc{D}_{\widetilde{X}}$-modules $j_!\mc{O}_U$ and $j_+\mc{O}_U$, it is necessary to introduce the maximal extension functor 
\[
\Xi_\rho: \mc{M}_\mathrm{hol}(\mc{D}_U) \rightarrow \mc{M}_\mathrm{hol}(\mc{D}_{\widetilde{X}}).
\]
This functor (defined in \eqref{eq: maximal extension def} below) extends $j_+$ and $j_!$ (see \eqref{eq: canonical ses 1} - \eqref{eq: canonical ses 2}), so it is a natural way to study both modules $j_!\mc{O}_U$ and $j_+\mc{O}_U$ at once. In this section, we give the construction of $\Xi_\rho$, then describe the $\widetilde{\mc{U}}$-module structure on $\Gamma(\widetilde{X}, \Xi_\rho\mc{O}_U)$. 

To start, we recall the construction of maximal extension for $\mc{D}$-modules, which is a special case of the construction in \cite{Bei87}, which produces the maximal extension and nearby cycles functors. Let $Y$ be a smooth variety, $f:Y \rightarrow \mathbb{A}^1$ a regular function, and 
\begin{equation}
    \label{eq: open closed decomp}
U:=f^{-1}(\mathbb{A}^1 - \{0\}) \xhookrightarrow{j} Y \xhookleftarrow{i} f^{-1}(0) 
\end{equation}
the corresponding open-closed decomposition of $Y$. For $n \in \mathbb{N}$, denote by 
\begin{equation}
    \label{eq: In}
    I^{(n)}:= \left( \mc{O}_{\mathbb{A}^1 - \{0\}} \otimes \C[s]/s^n \right) t^s
\end{equation}
the free rank $1$ $ \mc{O}_{\mathbb{A}^1 - \{0\}} \otimes \C[s]/s^n$-module generated by the symbol $t^s$. The action $\partial_t \cdot t^s = st^{-1}t^s$ gives $I^{(n)}$ the structure of a $\mc{D}_{\mathbb{A}^1-\{0\}}$-module. Any $\mc{D}_U$-module $\mc{M}_U$ can be deformed using $I^{(n)}$: set 
\begin{equation}
    \label{eq: fsMu}
    f^s\mc{M}_U^{(n)}:= f^+I^{(n)} \otimes_{\mc{O}_U} \mc{M}_U
\end{equation}
to be the $\mc{D}_U$-module obtained by twisting $\mc{M}_U$ by $I^{(n)}$. Note that $f^s\mc{M}_U^{(1)} = \mc{M}_U$, and that both $I^{(n)}$ and $f^s\mc{M}_U^{(n)}$ have a natural action by $s \in \C[s]/s^n$.

Assume that $\mc{M}_U$ is holonomic. Denote by 
\begin{equation}
    \label{eq: canonical map}
    \mathrm{can}: j_! f^s \mc{M}_U^{(n)} \rightarrow j_+f^s \mc{M}_{U}^{(n)}
\end{equation}
the canonical map between $!$ and $+$ pushforward, and 
\begin{equation}
\label{eq: s1n}
    s^1(n): j_! f^s \mc{M}_U^{(n)} \rightarrow j_+f^s \mc{M}_{U}^{(n)}
\end{equation}
the composition of $\mathrm{can}$ with multiplication by $s$. For large enough $n$, the cokernel of $s^1(n)$ stabilizes; i.e., $\coker s^1(n) = \coker s^1(n+k)$ for all $k>0$. For $n\gg 0$, define the $\mc{D}_Y$-module 
\begin{equation}
    \label{eq: maximal extension def}
    \Xi_f \mc{M}_U := \coker s^1(n),
\end{equation}
called the {\em maximal extension} of $\mc{M}_U$. By construction, this module comes equipped with the nilpotent endomorphism $s$. The corresponding functor
\begin{equation}
    \label{eq: maximal extension functor}
    \Xi_f: \mc{M}_\mathrm{hol}(\mc{D}_U) \rightarrow \mc{M}_{\mathrm{hol}}(\mc{D}_Y)
\end{equation}
is exact \cite[Lemma 4.2.1(i)]{BBJantzen}. Moreover, there are canonical short exact sequences \cite[Lemma 4.2.1 (ii)']{BBJantzen}
\begin{align}
    \label{eq: canonical ses 1}
    &0 \rightarrow j_!\mc{M}_U \rightarrow \Xi_f\mc{M}_U \rightarrow \coker(\mathrm{can}) \rightarrow 0 \\ 
    \label{eq: canonical ses 2}
    &0 \rightarrow \coker(\mathrm{can}) \rightarrow \Xi_f \mc{M}_U \rightarrow j_+ \mc{M}_U \rightarrow 0
\end{align}
with $j_! = \ker (s: \Xi_f  \rightarrow \Xi_f)$ and $j_+ = \coker(s: \Xi_f \rightarrow \Xi_f)$.

Now, we apply this general construction in the setting of our example. Let $\widetilde{X}$ and $U$ be as above (see \eqref{eq: base affine space is C2} and \eqref{eq: open B orbit U}), and let $f_\rho$ be the function\footnote{This choice of function corresponds to the deformation direction $\rho \in \h^*$, see Remarks \ref{rem: deformation direction algebraic} and \ref{rem: deformation direction geometric}.}. 
\begin{equation}
    \label{eq: our f}
    f_\rho: \widetilde{X} \rightarrow \mathbb{A}^1; (x_1, x_2) \mapsto x_2.
\end{equation}
For a variety $Y$, set 
\begin{equation}
    \label{eq: Ay}
    \mc{A}_Y:= \mc{D}_Y \otimes \C[s]/s^n.
\end{equation}
We will compute the maximal extension $\Xi_{\rho}\mc{O}_U:= \Xi_{f_\rho} \mc{O}_U$ of the structure sheaf $\mc{O}_U$ using the construction above, then describe the $\widetilde{\mc{U}}$-module structure on its global sections. To clarify the exposition, we list each step as a subsection. 

\subsubsection{Step 1: Deformation} 
\label{sec: step 1} Let $I^{(n)}$ be as in \eqref{eq: In}. The deformed version of $\mc{O}_U$ is 
\begin{equation}
    \label{eq: deformed OU}
    f^s \mc{O}_U^{(n)} = f^+I^{(n)} = \mc{O}_U \otimes_{f^{-1}(\mc{O}_{\mathbb{A}^1 - \{0\}})} f^{-1}(I^{(n)}). 
\end{equation}
The global sections of $f^s \mc{O}_U^{(n)}$ are 
\begin{equation}
\label{eq: v1 of pullback of O}
(\C[x_1, x_2, x_2^{-1}] \otimes \C[s]/s^n )t^s,
\end{equation}
where the differentials $\partial_1,\partial_2 \in \Gamma(\widetilde{X}, \mc{D}_{\widetilde{X}})$ act on the generator $t^s$ by  
\begin{equation}
    \label{eq: d2 acts on ts}
\partial_1 \cdot t^s = 0  \text{ and }  \partial_2 \cdot t^s = sx_2^{-1} t^s.
\end{equation}
 Alternatively, we can identify $f^s \mc{O}_U^{(n)}$ with a quotient of $\mc{A}_U$: 
\begin{equation}
\label{eq: identification of pullback with quotient}
    f^s \mc{O}_U ^{(n)} = \mc{A}_U / \mc{A}_U \langle \partial_1, x_2 \partial_2  - s\rangle.
\end{equation}
Both descriptions will be useful below.

\subsubsection{Step 2: $+$-pushforward} 
\label{sec: step 2} Because $j: U \hookrightarrow \widetilde{X}$ is an open embedding, the $\mc{D}_{\widetilde{X}}$-module  $j_+ f^s \mc{O}_U^{(n)}$ is the sheaf $f^s \mc{O}_U^{(n)}$ with $\mc{D}_{\widetilde{X}}$-module structure given by restriction to  $\mc{D}_{\widetilde{X}} \subset \mc{D}_U$. Under the identification \eqref{eq: identification of pullback with quotient}, we have
\begin{equation}
    j_+ f^s \mc{O}_U ^{(n)} = \mc{A}_U / \mc{A}_U \langle \partial_1, x_2 \partial_2 - s \rangle,
\end{equation}
with $\mc{D}_{\widetilde{X}}$-action given by left multiplication. 

It is interesting to examine the $\widetilde{\mc{U}}$-module structure on the global sections of this module. The operators $L_e, L_f, L_h$ and $R_h$ \eqref{eq: differential operators corresponding to e, f, h}, \eqref{eq: right action of h} act on the monomial basis elements  of \eqref{eq: v1 of pullback of O} by the following formulas 
\begin{align}
    \label{eq: action on deformed dual Vermas 1}
    L_e \cdot x_1^k x_2^\ell s^m t^s &= -k x_1^{k-1} x_2^{\ell+1} s^m t^s; \\
        \label{eq: action on deformed dual Vermas 2}
    L_f \cdot x_1^k x_2^\ell s^m t^s &= (-s -\ell) x_1^{k+1} x_2^{\ell-1} s^m t^s;\\
        \label{eq: action on deformed dual Vermas 3}
    L_h \cdot x_1^k x_2^\ell s^m t^s &= (s -k+\ell) x_1^k x_2^\ell s^m t^s;\\
        \label{eq: action on deformed dual Vermas 4}
    R_h \cdot x_1^k x_2^\ell s^m t^s &= (s +k + \ell) x_2^k x_2^\ell s^m t^s. 
\end{align}
The resulting $\widetilde{\mc{U}}$-module has a natural filtration given by powers of $s$, and it decomposes into a direct sum of submodules spanned by monomials $\{x_1^k x_2^\ell s^m t^s\}$ for fixed integers $k + \ell$. Each of these submodules has the structure of a deformed dual Verma module, as illustrated\footnote{We omit the generator $t^s$ and the arrows corresponding to the $R_h$-action in Figure \ref{fig: deformed dual verma} for clarity.} in Figure \ref{fig: deformed dual verma} for $k + \ell = 0$. 

\begin{figure}
    \centering
\makebox[\textwidth][c]{\includegraphics[width=1.2\textwidth]{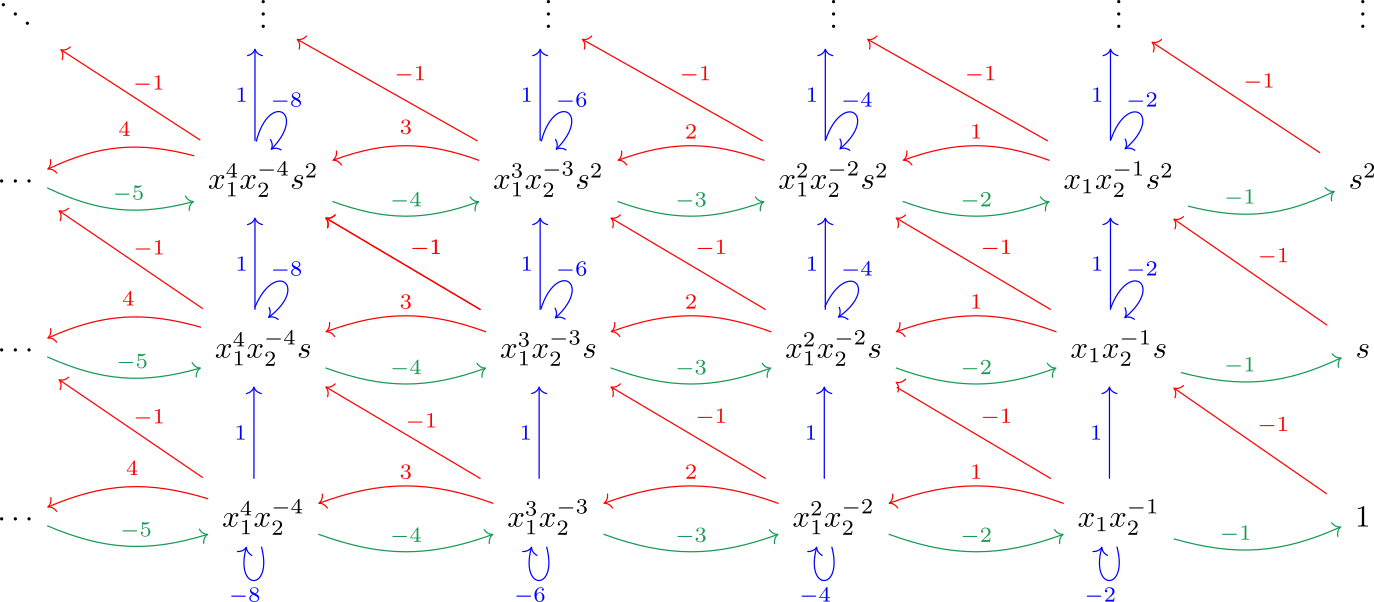}}
    \caption{Deformed dual Verma modules arise as global sections of $j_+f^s\mc{O}_U^{(n)}$.}
    \label{fig: deformed dual verma}
\end{figure}

Moreover, one can compute that the Casimir element $L_\Omega$ \eqref{eq: image of casimir under L} acts by
\begin{equation}
    \label{eq: Casimir action on deformed dual vermas}
    L_\Omega \cdot x_1^k x_2^\ell s^m t^s = \left( (k + \ell)^2 + 2(k + \ell) + 2s(1 + k + \ell) + s^2 \right) x_1^k x_2^\ell s^m t^s.
\end{equation}
Since $s$ is nilpotent, we can see from this computation that a high enough power of the operator
\begin{equation}
    \label{eq: generalised casimir action}
    L_\Omega - \gamma_{\mathrm{HC}}(k + \ell) = 2s(1 + k + \ell) + s^2
\end{equation}
annihilates any monomial basis element. (Here $\gamma_{\mathrm{HC}}$ is the Harish-Chandra projection in \eqref{eq: HC projection}.) Hence, the global sections of the submodules of $j_+ f^s \mc{O}_U ^{(n)}$ spanned by monomials $\{x_1^k x_2^\ell s^m t^s\}$ for fixed integers $k + \ell$  have generalised, but not strict, infinitesimal character. 

\subsubsection{Step 3: $!$-pushforward}
\label{sec: step 3} Recall that $j_! = \mathbb{D}_{\widetilde{X}} \circ j_+ \circ \mathbb{D}_U$, where $\mathbb{D}$ denotes holonomic duality, as in \eqref{eq: duality functor}. We begin by computing the right $\mc{D}_U$-module $\mathbb{D}_U f^s \mc{O}_U^{(n)}$ by taking a projective resolution of $f^s \mc{O}_U^{(n)}$ as a left $\mc{A}_U$-module. This is straightforward using the description \eqref{eq: v1 of pullback of O}. The complex
\begin{equation}
    \label{eq: projective resolution of fsO}
    0 \xrightarrow{d_2} \mc{A}_U \xrightarrow{d_1} \mc{A}_U \oplus \mc{A}_U \xrightarrow{d_0}  \mc{A}_U \xrightarrow{\epsilon} \mc{A}_U / \mc{A}_U \langle \partial_1, x_2 \partial_2 - s \rangle \rightarrow 0
\end{equation}
where $\epsilon$ is the canonical quotient map, $d_0$ sends $(\theta_1, \theta_2) \in \mc{A}_U \oplus \mc{A}_U$ to $\theta_1 \partial_1 - \theta_2 (x_2 \partial_2 -s)$, and $d_1$ sends $1 \mapsto (x_2 \partial_2 - s, \partial_1)$, is a free resolution of the left $\mc{A}_U$-module $f^s \mc{O}_U^{(n)}$. Applying the functor $\mathrm{Hom}_{\mc{A}_U}( - , \mc{A}_U)$ and making the natural identification    
\begin{equation}
\label{eq: hom identification}
    \mathrm{Hom}_{\mc{A}_U}(\mc{A}_U, \mc{A}_U) \simeq \mc{A}_U; 
    \varphi \mapsto \varphi(1),
\end{equation}
of right $\mc{A}_U$-modules, we see that 
\begin{equation}
    \label{eq: Ext}
\mathbb{D}_U f^s \mc{O}_U^{(n)} = \mathrm{Ext}_{\mc{A}_U}^2(f^s \mc{O}_U^{(n)}, \mc{A}_U) = \im d_1^* \backslash \ker d_2^*  = \langle \partial_1, x_2 \partial_2 - s \rangle \mc{A}_U \backslash \mc{A}_U.
\end{equation}
Here $d_i^* (\varphi) = \varphi \circ d_i$ for an appropriate homomorphism $\varphi$, and the right $\mc{A}_U$-module structure is given by right multiplication. 

To finish the computation of $j_! f^s \mc{O}_U^{(n)}$ we must take a projective resolution of this module. We do so following a similar process to the $!$-pushforward computation in Section \ref{sec: vermas and dual vermas}. Denote by $I$ the right ideal $\langle \partial_1, x_2 \partial_2 - s \rangle \mc{A}_U$ in $\mc{A}_U$. The complex 
\begin{equation}
    \label{eq: second proj resolution deformed}
    0 \leftarrow I \backslash \mc{A}_U \xleftarrow{\epsilon} \mc{A}_{\widetilde{X}} \xleftarrow{d_0} \mc{A}_{\widetilde{X}} \oplus \mc{A}_{\widetilde{X}} \xleftarrow{d_1} \mc{A}_{\widetilde{X}} \xleftarrow{d_2} 0
\end{equation}
with maps given by 
\begin{align}
    \label{eq: maps of second proj resolution deformed}
    \epsilon: 1 &\mapsto Ix_2^{-1}; \\
    d_0: (\theta_1, \theta_2) &\mapsto x_2 \partial_1 \theta_1 - (x_2^2 \partial_2 - x_2 s) \theta_2; \\
    d_1: 1 &\mapsto (x_2 \partial_2 - s, \partial_1)
\end{align}
is a free resolution of $\mathbb{D}_U f^s \mc{O}_U^{(n)}$ by right $\mc{A}_{\widetilde{X}}$-modules. Applying $\mathrm{Hom}_{\mc{A}_{\widetilde{X}}, r}( - , \mc{A}_{\widetilde{X}})$ and making the natural identifications as above, we obtain 
\begin{equation}
    \label{eq: shriek pushforward deformed}
    j_! f^s \mc{O}_U^{(n)} = \ker d_2^* / \im d_1^* = \mc{A}_{\widetilde{X}} / \mc{A}_{\widetilde{X}} \langle \partial_1, x_2 \partial_2 - s \rangle.
\end{equation}
The left $\mc{A}_{\widetilde{X}}$-module structure is given by left multiplication. 

Again, it is interesting to examine the $\widetilde{\mc{U}}$-module structure on the global sections of this module. The global sections of $j_! f^s \mc{O}_U^{(n)}$ are spanned by monomials $x_1^k x_2^{\ell} s^m$ for $k, \ell \geq 0$ and $0 \leq m < n$ and $x_1^a \partial_2^ b s^m$ for $a, b \geq 0$ and $0 \leq m < n$. For $\ell>0$, the $L_e, L_f, L_h$ and $R_h$-actions on the monomials $x_1^k x_2^{\ell} s^m$ are as in \eqref{eq: action on deformed dual Vermas 1} - \eqref{eq: action on deformed dual Vermas 4} (where we identify the generator $t^s$ of $j_+ f^s \mc{O}_U^{(n)}$ with the coset containing $1$ in $j_!f^s \mc{O}_U^{(n)}$), and the actions on the monomials $x_1^a \partial_2^ b s^m$ are given by the following formulas: 

\begin{align}
    \label{eq: action on deformed Vermas 1}
    L_e \cdot x_1^a \partial_2^ b s^m &= a(b - 1 - s) x_1^{a-1} \partial_2 ^{b-1} s^m; \\
       \label{eq: action on deformed Vermas 2}
    L_f \cdot x_1^a \partial_2^ b s^m &= -x_1^{a+1} \partial_2^{b+1} s^m;\\
       \label{eq: action on deformed Vermas 3}
    L_h \cdot x_1^a \partial_2^ b s^m &= (s - a - b) x_1^a \partial_2^b s^m;\\
       \label{eq: action on deformed Vermas 4}
    R_h \cdot x_1^a \partial_2^ b s^m &= (s+ a - b) x_1^a \partial_2^b s^m.
\end{align}
For $\ell =b= 0$, the actions of $L_h$ and $R_h$ are as in \eqref{eq: action on deformed Vermas 3}-\eqref{eq: action on deformed Vermas 4}, and the actions of $L_e$ and $L_f$ are given by 
\begin{align}
    \label{eq: deformed actions at 0}
    L_e \cdot x_1^k s^m &= -k x_1^{k-1} x_2 s^m; \\
    L_f \cdot x_1^k s^m &= -x_1^{k+1} \partial_2 s^m.
\end{align}
 As in \ref{sec: step 2}, this $\widetilde{\mc{U}}$-module has an $n$-step filtration by powers of $s$, and decomposes into a direct sum of $\widetilde{\mc{U}}$-submodules, each spanned by the set of monomials $\{x_1^k x_2^{\ell} s^m\}$ and $\{x_1^a \partial_2^ b s^m\}$ such that $k + \ell = a-b$ is a fixed integer. For $k + \ell = a-b=\lambda$, this submodule is isomorphic to a deformed Verma module of highest weight $\lambda$. We illustrate the module corresponding to $\lambda = 0$ in Figure \ref{fig: deformed verma}

  \begin{figure}
    \centering
\makebox[\textwidth][c]{\includegraphics[width=1.2\textwidth]{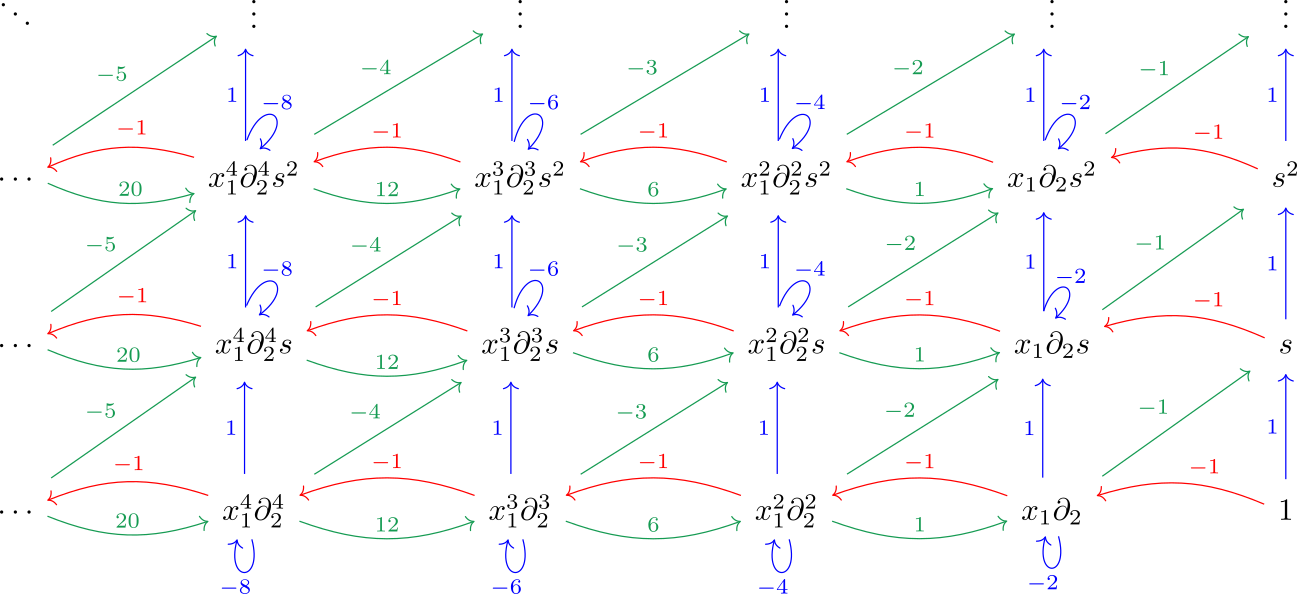}}
    \caption{Deformed Verma modules arise as global sections of $j_!f^s\mc{O}_U^{(n)}$.}
    \label{fig: deformed verma}
\end{figure}

\subsubsection{Step 4: Image of the canonical map} 
\label{sec: step 4} 
Set $I_U = \mc{A}_U \langle \partial_1, x_2 \partial_2 - s \rangle$ and $I_{\widetilde{X}} = \mc{A}_{\widetilde{X}} \langle \partial_1, x_2 \partial_2 - s \rangle$ to be the left ideals generated by the operators $\partial_1$ and $x_2 \partial_2 - s$ in $\mc{A}_U$ and $\mc{A}_{\widetilde{X}}$, respectively. The canonical map between $!$- and $+$-pushforward is given by 
\begin{align*}
    j_! f^s \mc{O}_U^{(n)} = \mc{A}_{\widetilde{X}} / I_{\widetilde{X}} &\xrightarrow{\mathrm{can}}  \mc{A}_U / I_U =j_+ f^s \mc{O}_U^{(n)} \\
    1I_{\widetilde{X}} &\longmapsto 1 I_U
\end{align*}
Since $1\mc{A}_{\widetilde{X}}$ generates $j_! f^s \mc{O}_U^{(n)}$ as a $\mc{A}_{\widetilde{X}}$-module, its image completely determines the morphism $\mathrm{can}$. On the monomial basis elements $x_1^k x_2^\ell s^m$ and $x_1^k \partial_2^\ell s^m$ of $j_! f^s \mc{O}_U^{(n)}$, the canonical map acts by 
\begin{equation}
\label{eq: canonical map on monomial basis}
x_1^k x_2^\ell s^m \xmapsto{\mathrm{can}} x_1^k x_2^\ell s^m \text{ and }
x_1^a \partial_2^b s^m \xmapsto{\mathrm{can}} s(s-1)\cdots (s-b+1) x_1^a x_2^{-b} s^m
\end{equation}
for $b>1$. For $b=1$, $x_1^a \partial_2 s^m \xmapsto{\mathrm{can}} sx_1^a x_2^{-1}$.  

The image of the morphism $\mathrm{can}$ is the $\mc{A}_{\widetilde{X}}$-submodule 
\begin{equation}
\label{eq: image of can}
\im (\mathrm{can}) = \mc{A}_{\widetilde{X}}/I_U \subset \mc{A}_U / I_U. 
\end{equation}
In the description of the global sections of $j_+ f^s \mc{O}_U^{(n)}$ in \eqref{eq: v1 of pullback of O}, the global sections of $\im (\mathrm{can})$ can be identified with 
\begin{equation}
    \label{eq: global sections of image of can}
    \left(\C[x_1, x_2] \otimes \C[s]/s^n + \C[x_1, x_2, x_2^{-1}] \otimes s \C[s]/s^n\right)t^s.
\end{equation}

\subsubsection{Step 5: The maximal extension}
\label{sec: step 5} Composing the canonical map $\mathrm{can}$ with $s$ gives 
\begin{equation}
\label{eq: s1n}
s^1(n): \mc{A}_{\widetilde{X}}/I_{\widetilde{X}} \xrightarrow{\mathrm{can}} \mc{A}_U/I_U \xrightarrow{s} \mc{A}_U/I_U.
\end{equation}
The global sections of the image of $s^1(n)$ (as a submodule of \eqref{eq: v1 of pullback of O}) are 
\begin{equation}
    \label{eq: global sections of image of s1n}
    \Gamma(\widetilde{X}, \im s^1(n)) \simeq \left( \C[x_1, x_2] \otimes s\C[s]/s^n + \C[x_1, x_2, x_2^{-1}] \otimes s^2 \C[s]/s^n \right)t^s. 
\end{equation}
This gives us an explicit description of $\Xi_\rho \mc{O}_U = \coker s^1(n)$:
\begin{align}
\label{eq: global sections of maximal extension}
    \Gamma(\widetilde{X}, \Xi_\rho \mc{O}_U) &= \left(\C[x_1, x_2, x_2^{-1}] \otimes \C[s]/s^n \right) t^s / \Gamma(\widetilde{X}, \im s^1(n)) \\
    \label{eq: global sections of maximal extension 2}
    &= \left( \C[x_1, x_2, x_2^{-1}] \otimes \C[s]/s^2 \right)t^s / \left( \C[x_1, x_2]\otimes s \C[s]/s^2 \right)t^s
\end{align}

\begin{figure}
    \centering
\makebox[\textwidth][c]{\includegraphics[width=1.2\textwidth]{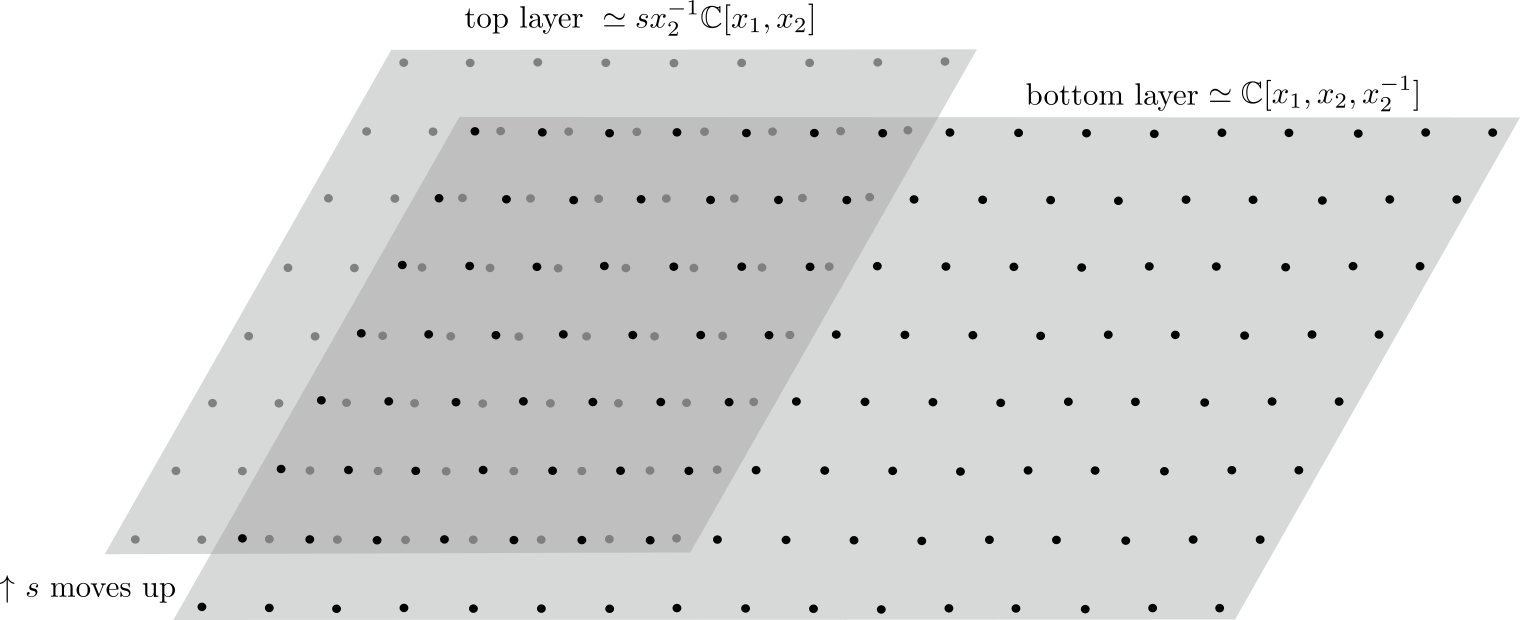}} 
    \caption{Caricature of the maximal extension $\Xi_\rho \mc{O}_U$. }
    \label{fig: layers of maximal extension}
\end{figure}

A caricature of the  $\Gamma(\widetilde{X}, \mc{A}_{\widetilde{X}})$-module \eqref{eq: global sections of maximal extension 2} is illustrated as in Figure \ref{fig: layers of maximal extension}. It has two layers, corresponding to the two non-zero powers of $s$, and action by $s$ moves layers up. As vector spaces, the bottom layer is isomorphic to $\C[x_1, x_2, x_2^{-1}]$ and the top layer to $sx_2^{-1} \C[x_1, x_2^{-1}]$.

Our final step is to examine the $\widetilde{\mc{U}}$-module structure on $\Gamma(\widetilde{X}, \Xi_\rho\mc{O}_U)$. The module \eqref{eq: global sections of maximal extension 2} has a basis given by monomials $x_1^k x_2^\ell t^s$ for $k \in \Z_{\geq 0}$ and $\ell \in \Z$ and $x_1^k x_2^\ell s t^s$ for $k \in \Z_{\geq 0}$ and $\ell \in \Z_{< 0}$. The actions of the operators $L_e, L_f, L_h$ and $R_h$ \eqref{eq: differential operators corresponding to e, f, h} on these monomials are given by applying the formulas \eqref{eq: action on deformed dual Vermas 1}- \eqref{eq: action on deformed dual Vermas 4} and taking the image of the resulting monomials in the quotient \eqref{eq: global sections of maximal extension 2}. 

The $\widetilde{\mc{U}}$-module $\Gamma(\widetilde{X}, \Xi_\rho \mc{O}_U)$ splits into a direct sum of submodules spanned by monomials $x_1^k x_2^\ell t^s$ and $x_1^k x_2^\ell s t^s$ such that $k + \ell$ is a fixed integer. We illustrate the submodule for $k + \ell = 0$ in Figure \ref{fig: maximal extension}\footnote{For clarity, we drop the generator $t^s$ from our notation in Figure \ref{fig: maximal extension}.}. If $\lambda\geq 0$, the submodule corresponding to the integer $\lambda = k + \ell$ has the Verma module of highest weight $\lambda$ as a submodule, and the dual Verma module corresponding to $\lambda$ as a quotient. As a $\mc{U}(\mf{g})$-module, it is isomorphic to the big projective module\footnote{The big projective module is the projective cover of the irreducible highest weight module $L(w_0 \lambda)$, where $w_0$ is the longest element of the Weyl group. It is the longest indecomposable projective object in the block $\mc{O}_\lambda$ of category $\mc{O}$ \cite[\S3.12]{BGGcatO}.} $P(w_0\lambda)$ in the corresponding block of category $\mc{O}$.  

\begin{figure}
    \centering
\makebox[\textwidth][c]{\includegraphics[width=1.2\textwidth]{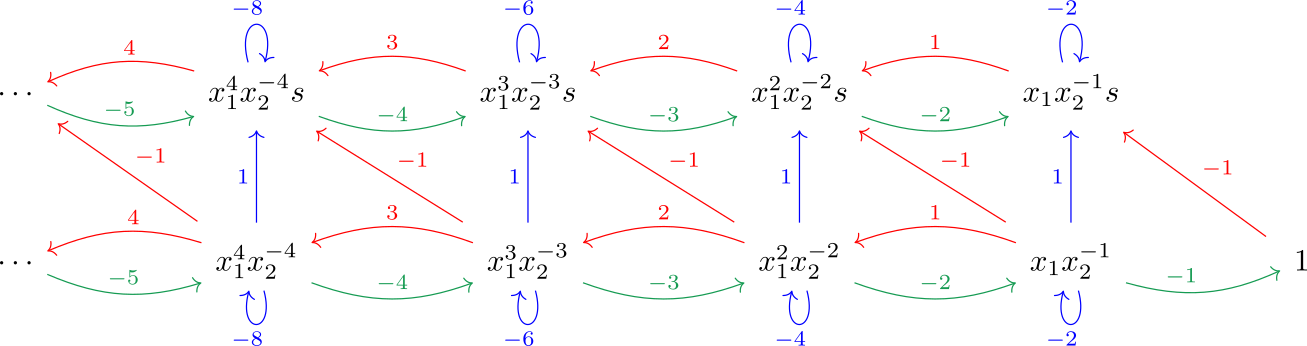}}    \caption{Big projective modules arise as global sections of slices of $\Xi_\rho \mc{O}_U$.}
    \label{fig: maximal extension}
\end{figure}

\subsection{The monodromy filtration and the geometric Jantzen filtration}
\label{sec: monodromy filtration and Jantzen filtration}

The maximal extension $\Xi_\rho \mc{O}_U$ naturally comes equipped with a nilpotent endomorphism $s$, giving it a corresponding monodromy filtration. This is the source of the geometric Jantzen filtrations on $j_!\mc{O}_U$ and $j_+\mc{O}_U$. In this section, we use the monodromy filtration on $\Xi_\rho \mc{O}_U$ to compute the geometric Jantzen filtration on $j_!\mc{O}_U$. Using the computations of Section \ref{sec: maximal extension}, we then describe the corresponding $\widetilde{\mc{U}}$-module filtration on global sections.

We begin by recalling monodromy filtrations in abelian categories, following \cite[\S1.6]{Del}. Given an object $A$ in an abelian category $\mc{A}$ and a nilpotent endomorphism $s: A \rightarrow A$, it follows from the Jacobson-Morosov theorem \cite[Proposition 1.6.1]{Del} that there exists a unique increasing exhaustive filtration $\mu^\bullet$ on $A$ such that $s\mu^n \subset \mu^{n-2}$, and for $k \in \mathbb{N}$, $s^k$ induces an isomorphism $\mathrm{gr}_\mu^k A \simeq \mathrm{gr}_\mu^{-k} A.$ This unique filtration is called the {\em monodromy filtration} of $A$.

Following Deligne's proof in \cite[\S1.6]{Del}, the monodromy filtration can be described explicitly in terms of powers of $s.$ Namely, if we set
\begin{equation}
    \label{eq: kernel filtration}
    \mathscr{K}^p A: = \begin{cases} \ker s^{p+1} & \text{for } p \geq 0; \\ 0, & \text{for } p<0 \end{cases}
\end{equation}
to be the kernel filtration of $A$ and 
\begin{equation}
    \label{eq: image filtration}
    \mathscr{I}^q A:= \begin{cases} \im s^q & \text{for } q>0; \\ A & \text{for } q \leq 0, \end{cases}
\end{equation}
to be the image filtration of $A$, then $\mu^\bullet$ is the convolution of the kernel and image filtrations; i.e.,
\begin{equation}
	\label{eq: monodromy filtration}
\mu^r = \sum_{p-q = r} \mathscr{K}^p \cap \mathscr{I}^q.
\end{equation}

The monodromy filtration $\mu^\bullet$ induces filtrations $J_!^\bullet$ on $\ker s$ and $J_+^\bullet$ on $\coker s$. By (\ref{eq: monodromy filtration}), these can be seen to be
\begin{equation}
    \label{eq: geometric Jantzen filtration} J_!^i = \ker s \cap \mathscr{I}^{-i} \text{ and } J_+^i = (\mathscr{K}^i + \im s) / \im s. 
\end{equation}
In the setting of holonomic $\mc{D}$-modules, the filtrations $J_!^\bullet$ and $J_+^\bullet$ define the geometric Jantzen filtrations. 
\begin{definition}
    \label{def: geometric Jantzen filtration} Let $Y$ be a smooth variety, $f$ a regular function on $Y$ and $U=f^{-1}(\mathbb{A}^1 - \{0\})$ as in \eqref{eq: open closed decomp}. For a holonomic $\mc{D}_U$-module $\mc{M}_U$, recall that $j_! \mc{M}_U = \ker(s: \Xi_f \mc{M}_U \rightarrow \Xi_f \mc{M}_U)$ and $j_+ \mc{M}_U= \coker(s: \Xi_f \mc{M}_U \rightarrow \Xi_f \mc{M}_U)$ \cite[Lemma 4.2.1]{BBJantzen}.  The filtrations $J_!^\bullet$ of $j_!\mc{M}_U$ and $J_+^\bullet$ of $j_+ \mc{M}_U$ are called the {\em geometric Jantzen filtrations}.
\end{definition}

Now we return to our running example. The monodromy filtration $\mu^\bullet$ on $\Xi_\rho \mc{O}_U$ is 
\begin{equation}
\label{eq: monodromy filtration on maximal extension}
   \mu^{-2} = 0 \subset \mu^{-1} = \im s \subset \mu^0 =  j_! \mc{O}_U \subset \mu^1  =\Xi_\rho \mc{O}_U.
\end{equation}
Restricting this to $\ker s = j_! \mc{O}_U$, we obtain the geometric Jantzen filtration of $j_! \mc{O}_U$:
\begin{equation}
    \label{eq: geometric Jantzen filtration of j!}
    0 \subset \im s \subset j_! \mc{O}_U.
\end{equation}
The induced filtration on $\coker s = j_+ \mc{O}_U$ gives the geometric Jantzen filtration on $j_+ \mc{O}_U$:
\begin{equation}
    0 \subset \ker s / \im s \subset j_+ \mc{O}_U. 
\end{equation}

\begin{remark}
    \label{rem: deformation direction geometric} (Geometric deformation direction) There are other choices of regular functions on $\widetilde{X}$ which we could have used in the construction of these filtrations. In particular, if $\gamma \in \h^*$ is dominant and integral such that $\gamma(h)=n$ for $n \in \Z_{> 0}$, then the function $f_{\gamma}:(x_1, x_2) \mapsto x_2^n$ can be used to define an intermediate extension functor $\Xi_{f_\gamma}$ and corresponding Jantzen filtrations.  Beilinson--Bernstein establish that all such $f_\gamma$ lead to the same filtration. For general Lie algebras $\mf{g}$, the construction can also be done for other choices of meromorphic functions on $\widetilde{X}$, but it is unclear geometrically whether these result in different filtrations \cite[\S4.3]{BBJantzen}. This is comparable to the dependence on deformation direction in the algebraic Jantzen filtration, see Remark \ref{rem: deformation direction algebraic}.
\end{remark}

Using the computations in Section \ref{sec: maximal extension}, we can examine the $\widetilde{\mc{U}}$-module filtrations that we obtain on global sections. Recall that $\Gamma(\widetilde{X}, \Xi_\rho \mc{O}_U)$ decomposes into a direct sum of submodules spanned by monomials $x_1^k x_2^\ell t^s$ and $x_1^k x_2^\ell s t^s$ such that $k + \ell$ is a fixed non-negative integer. Figure \ref{fig: maximal extension} illustrates the submodule corresponding to $k + \ell = 0$. Looking at this figure, it is clear that $\ker s = \mathrm{span}\{ x_1^k x_2^\ell s t^s, t^s \}$ is isomorphic to the Verma module of highest weight $0$, and $\coker s = \mathrm{span}\{x_1^k x_2^\ell t^s\}$ is isomorphic to the corresponding dual Verma module. Moreover, the global sections of the monodromy filtration on $\Xi_\rho \mc{O}_U$ restricted to the submodule corresponding to $k + \ell = \lambda$ is the composition series of the corresponding big projective module $P( w_0\lambda)$ when $\lambda \geq 0$. This is illustrated in Figure \ref{fig: global sections of monodromy} for $\lambda = 0$. We conclude that the filtrations on the Verma module $M(\lambda)$ and dual Verma module $I(\lambda)$ obtained by taking global sections of the geometric Jantzen filtrations are the composition series\footnote{This is an $\mf{sl}_2(\C)$ phenomenon. For larger groups this procedure will yield a filtration different from the composition series.}.

\begin{figure}
    \centering
\makebox[\textwidth][c]{\includegraphics[width=1.2\textwidth]{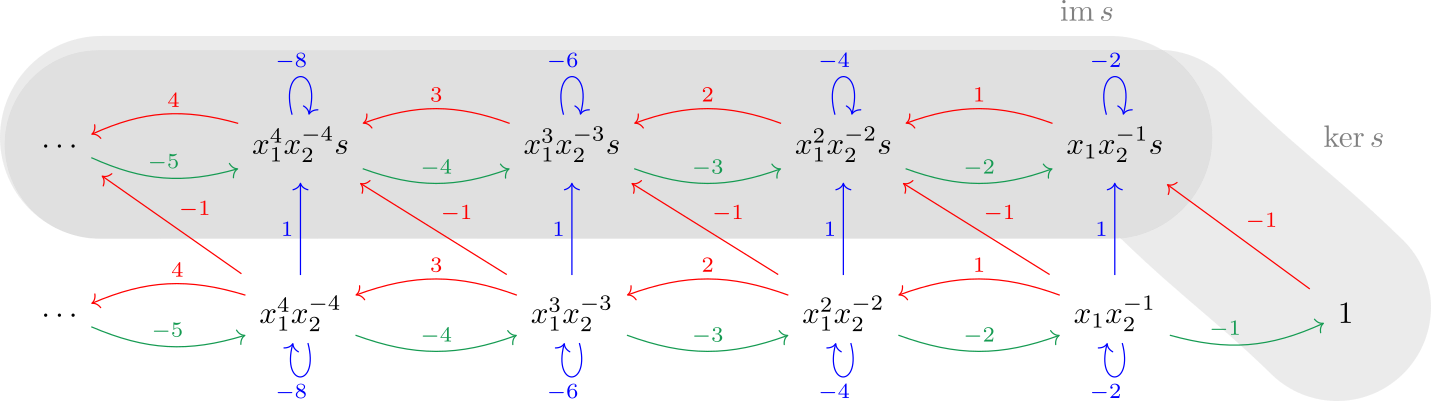}} 
    \caption{Global sections of the monodromy filtration on $\Xi_\rho \mc{O}_U$ are the composition series of the big projective module.}
    \label{fig: global sections of monodromy}
\end{figure}

\subsection{Relation to the algebraic Jantzen filtration}
\label{sec: relation to the algebraic Jantzen filtration}

The geometric Jantzen filtrations described above have an algebraic analogue, due to Jantzen \cite{Jantzen}. In this section, we recall the construction of the algebraic Jantzen filtration of a Verma module, then explain its  relation with the geometric construction in Section \ref{sec: monodromy filtration and Jantzen filtration}. 

\subsubsection{The algebraic Jantzen filtration}
\label{sec: the algebraic jantzen filtration}

We follow \cite{Soergel08}. Another nice reference for Jantzen filtrations is \cite{IK11}. 

Let $\mf{g}$ be a complex semisimple Lie algebra, $\mf{b}$ a fixed Borel subalgebra, $\mf{n} = [\mf{b}, \mf{b}]$ the nilpotent radical of $\mf{b}$, and $\mf{h}$ a Cartan subalgebra so that $\mf{b}= \mf{h} \oplus \mf{n}$. Denote by $\bar{\mf{b}}$ the opposite Borel subalgebra to $\mf{b}$. For $\lambda \in \mf{h}^*$, denote the Verma module of highest weight $\lambda$ by 
\begin{equation}
    \label{eq: Verma module def}
    M(\lambda) := \mc{U}(\mf{g}) \otimes_{\mc{U}(\mf{b})} \C_\lambda.
\end{equation}
Denote by $I(\lambda)$ the corresponding dual Verma module, defined to be the direct sum of weight spaces in
\begin{equation}
    \label{eq: dual verma def}
    \Hom_{\mc{U}(\overline{\mf{b}})} (\mc{U}(\mf{g}), \C_\lambda).
\end{equation}

Set $T = \mc{O}(\C\rho)$ to be the ring of regular functions on the line $\C \rho \subset \mf{h}^*$, where $\rho$ is the half sum of positive roots in the root system determined by $\mf{b}$. A choice of linear functional $s: \C \rho \rightarrow \C$ gives an isomorphism $T \simeq \C[s]$. Fix such an identification. Set $A:=T_{(s)}$ to be the local $\C$-algebra obtained from $T$ by inverting all polynomials not divisible by $s$, and 
\begin{equation}
    \label{eq: phi}
    \varphi: \mc{O}(\mf{h}^*) \rightarrow A
\end{equation} 
to be the composition of the restriction map $\mc{O}(\mf{h}^*) \rightarrow T$ with the inclusion $T \hookrightarrow A$. Note that under the identification $\mc{U}(\mf{h}) \simeq \mc{O}(\mf{h}^*)$, $\varphi(\mf{h}) \subseteq (s)$, the unique maximal ideal of $A$. 

Let $V$ be a $(\mf{g},A)$-bimodule on which the right and left actions of $\C$ agree. The {\em deformed weight space} $V^\mu$ of $V$ corresponding to a weight $\mu$ is the subspace 
\begin{equation}
    \label{eq: deformed weight space}
    V^\mu:= \{ v \in V \mid (h - \mu(h))v=v \varphi(h) \text{ for all }h \in \h \}.
\end{equation}
The direct sum of all deformed weight spaces of $V$ is a $(\mf{g},A)$-submodule of $V$ \cite[\S2.3]{Soergel08}. 

For $\lambda \in \mf{h}^*$, the {\em deformed Verma module} corresponding to $\lambda$ is the $(\mf{g},A)$-bimodule 
\begin{equation}
    \label{eq: deformed Verma definition}
    M_A(\lambda):= \mc{U}(\mf{g}) \otimes_{\mc{U}(\mf{b})} A_\lambda,
\end{equation}
where the $\mc{U}(\mf{b})$-module structure on $A_\lambda$ is given by extending the $\h$-action 
\begin{equation}
    \label{eq: h action in deformed verma}
    h \cdot a = (\lambda + \varphi)(h) a
\end{equation}
trivially to $\mf{b}$. Here $h \in \h$, $a \in A$, and $\varphi$ is as in \eqref{eq: phi}. The deformed Verma module $M_A(\lambda)$ is equal to the direct sum of its deformed weight spaces.

The {\em deformed dual Verma module} $I_A(\lambda)$ corresponding to $\lambda$ is the direct sum of deformed weight spaces in the $(\mf{g},A)$-bimodule
\begin{equation}
    \label{eq: deformed dual Verma definition}
     \Hom_{\mc{U}(\bar{\mf{b}})}(\mc{U}(\mf{g}), A_\lambda).
\end{equation}

There is a canonical isomorphism \cite[Proposition 2.12]{Soergel08}
\begin{equation}
    \label{eq: hom between verma and dual}
    \Hom_{(\mf{g},A) \mathrm{-mod}}(M_A(\lambda), I_A(\lambda)) \simeq A.
\end{equation}
Under this isomorphism, $1 \in A$ distinguishes a canonical $(\mf{g},A)$-bimodule homomorphism 
\begin{equation}
    \label{eq: canonical map algebraically}
    \psi_{A, \lambda}: M_{A}(\lambda) \rightarrow I_A(\lambda).
\end{equation}

For any $A$-module $M$, there is a descending $A$-module filtration $M^i:=s^iM$ with associated graded $gr^iM = M^i/M^{i+1}$. Hence there is a reduction map 
\begin{equation}
\label{eq: reduction map}
    \mathrm{red}: M \rightarrow gr^0M = M/sM.
\end{equation}

For $M_A(\lambda)$ and $I_A(\lambda)$, the layers of this filtration are $\g$-stable, so we obtain surjective $\g$-module homomorphisms 
\begin{equation}
    \label{eq: reduction map on vermas and dual vermas} 
    \mathrm{red}: M_A(\lambda) \rightarrow M(\lambda) = gr^0 M_A(\lambda) \text{ and } \mathrm{red}: I_A(\lambda) \rightarrow I(\lambda) = gr^0 I_A(\lambda).
\end{equation}
Pulling back the filtration above along the canonical map $\psi_{A, \lambda}$ \eqref{eq: canonical map algebraically} gives a $(\g, A)$-bimodule filtration of $M_A(\lambda)$.
\begin{definition}
    \label{def: algebraic Jantzen filtration}
    The {\em algebraic Jantzen filtration} of $M_A(\lambda)$ is the $(\g, A)$-bimodule filtration 
    \[
    M_A(\lambda)^i:= \left\{m \in M_A(\lambda) \mid \psi_{A, \lambda} (m) \in s^i I_A(\lambda) \right\},
    \]
    where $\psi_{A, \lambda}$ is the canonical map \eqref{eq: canonical map algebraically}. By applying the reduction map \eqref{eq: reduction map on vermas and dual vermas} to the filtration layers, we obtain a filtration $M(\lambda)^\bullet$ of $M(\lambda)$. 
\end{definition}

\subsubsection{Relationship between algebraic and geometric Jantzen filtrations}
\label{sec: relationship between algebraic and geometric Jantzen filtrations}

Though the constructions seem quite different at first glance, the geometric Jantzen filtration in Section \ref{sec: monodromy filtration and Jantzen filtration} aligns with the algebraic Jantzen filtration described in Section \ref{sec: the algebraic jantzen filtration} under the global sections functor. In this final section, we illustrate this relationship through our running example. 

Recall the canonical map \eqref{eq: canonical map}
\[
\mathrm{can}:j_! f^s \mc{O}_U^{(n)} \rightarrow j_+ f^s \mc{O}_U^{(n)}.
\]
As illustrated in Figures \ref{fig: deformed dual verma} and \ref{fig: deformed verma} , the global sections of $j_!f^s \mc{O}_U^{(n)}$ and $j_+ f^s \mc{O}_U^{(n)}$ decompose into direct sums of deformed dual Verma and Verma modules, respectively.\footnote{To be more precise, the submodules of $\Gamma(\widetilde{X}, j_!f^s \mc{O}_U^{(n)})$ and $\Gamma(\widetilde{X}, j_+ f^s \mc{O}_U^{(n)})$ corresponding to an integer $\lambda$ are truncated versions of $M_A(\lambda)$ \eqref{eq: deformed dual Verma definition} and $I_A(\lambda)$ \eqref{eq: deformed dual Verma definition} obtained by taking a quotient so that $s^n = 0$.}. The global sections of $\mathrm{can}$ are the direct sum of $\psi_{A,\lambda}$ \eqref{eq: canonical map algebraically} for all integral $\lambda$. There are two natural filtrations of $j_! \mc{O}_U$ which we have described using this set-up. 

\vspace{2mm}

\noindent
{\em Filtration 1:} (algebraic Jantzen filtration) 
\vspace{1mm}

We obtain a filtration of $j_! f^s \mc{O}_U^{(n)}$ by pulling back the ``powers of s'' filtration along $\mathrm{can}$. This induces a filtration on the quotient
\begin{equation}
    \label{eq: j! as a quotient}
    j_!(\mc{O}_U) \simeq j_! f^s \mc{O}_U^{(n)} / s j_! f^s \mc{O}_U^{(n)}.
\end{equation}
This is exactly the $\mc{D}$-module analogue of the algebraic Jantzen filtration described in Section \ref{sec: the algebraic jantzen filtration}. On global sections, it is the filtration 
\begin{equation}
    \label{eq: algebraic Jantzen on global sections}
    F^i \Gamma(\widetilde{X}, j_! \mc{O}_U) = \left\{ v \in \Gamma(\widetilde{X}, j_! \mc{O}_U) \mid \Gamma(\mathrm{can})(v) \in s^i \Gamma(\widetilde{X}, j_+ f^s \mc{O}_U^{(n)}) \right\}.
\end{equation}

\vspace{2mm}

\noindent
{\em Filtration 2:} (geometric Jantzen filtration) 

\vspace{1mm}

There is a unique monodromy filtration on $\Xi_\rho \mc{O}_U = \coker (s \circ \mathrm{can})$. Restricting this to the kernel of $s$, we obtain a filtration on 
\begin{equation}
\label{eq: j! as a sub}
j_! \mc{O}_U \simeq \ker (s: \Xi_\rho \mc{O}_U \rightarrow \Xi_\rho \mc{O}_U). 
\end{equation}
This is the geometric Jantzen filtration. It can be realized explicitly in terms of the image of powers of $s$ using \ref{eq: geometric Jantzen filtration}. On global sections, this gives 
\begin{equation}
    \label{eq: global sections of geometric Jantzen}
    G^i \Gamma(\widetilde{X}, j_! \mc{O}_U) = \left\{ w \in \ker (s \circlearrowright \Gamma( \widetilde{X}, \Xi_\rho \mc{O}_U)) \mid w \in s^i \Gamma( \widetilde{X}, \Xi_\rho \mc{O}_U) \right\}.
\end{equation}

It is helpful to see these filtrations in a picture. Figure \ref{fig: comparing filtrations}
 illustrates the set-up when restricted to the submodule corresponding to $\lambda = 0$. 

 \begin{figure}
     \centering
\makebox[\textwidth][c]{\includegraphics[width=1.2\textwidth]{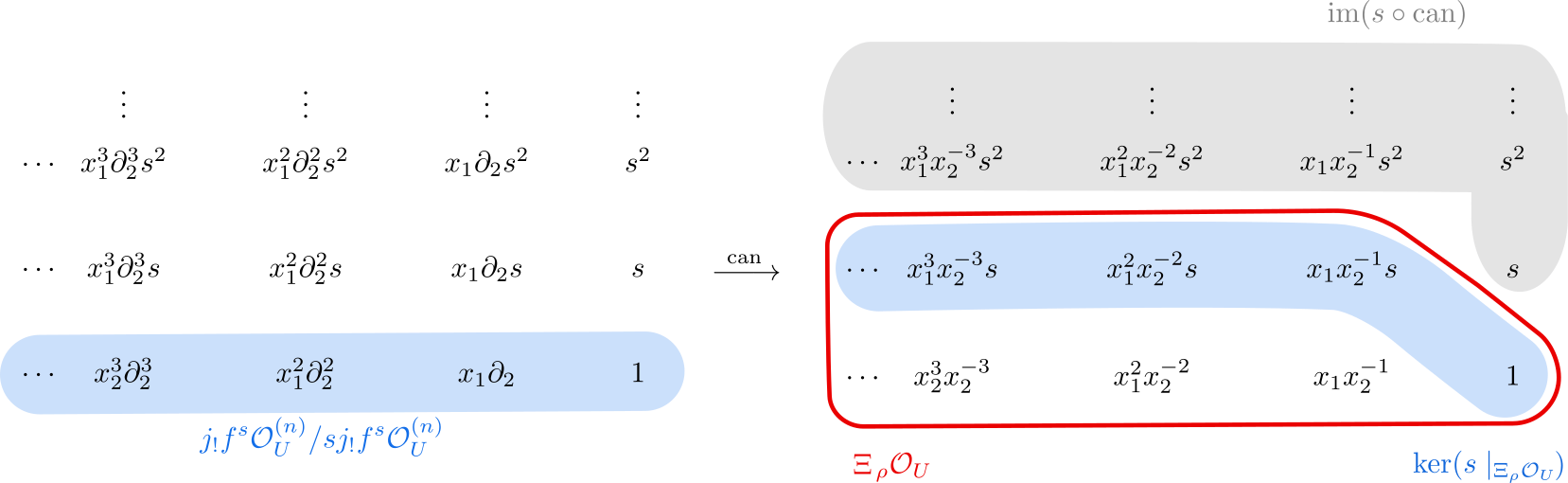}} 
     \caption{Relationship between the algebraic and geometric Jantzen filtrations}
     \label{fig: comparing filtrations}
 \end{figure}

 The map $\mathrm{can}$ is described on basis elements in \eqref{eq: canonical map on monomial basis}. Computing these actions for $\lambda = 0$, we see in Figure \ref{fig: comparing filtrations} that $\mathrm{can}$ fixes the right-most column and sends any other monomial on the left to a linear combination of monomials directly above the corresponding monomial on the right. The image of $s_1(n)=s \circ \mathrm{can}$ \eqref{eq: s1n} is highlighted in grey. The quotient by this image is the maximal extension, which is outlined in the red box. The quotient \ref{eq: j! as a quotient} is highlighted in blue in the left hand module, and the submodule \ref{eq: j! as a sub} is highlighted in blue in the right hand module. 

 We see that there are two copies of $j_! \mc{O}_U$ (each highlighted in blue in Figure \ref{fig: comparing filtrations}) in this set-up: one as a quotient of the left-hand module $j_! f^s \mc{O}_U^{(n)}$, and one as a submodule of a quotient of the right-hand module $j_+ f^s \mc{O}_U^{(n)}$. These two copies can be naturally identified as follows. 

 Because the submodule $s j_! f^s \mc{O}_U^{(n)}$ is in the kernel of the composition of $\mathrm{can}$ with the quotient $j_! f^s \mc{O}_U^{(n)} \rightarrow \coker( s \circ \mathrm{can}) = \Xi_\rho \mc{O}_U$, the map $\mathrm{can}$ descends to a map on the quotient: 
 \begin{equation}
     \label{eq: quotient of can}
     \overline{\mathrm{can}}: j_! \mc{O}_U \simeq j_! f^s \mc{O}_U ^{(n)} / s j_! f^s \mc{O}_U^{(n)} \rightarrow \Xi_\rho \mc{O}_U.
 \end{equation}
By construction, the map $\overline{\mathrm{can}}$ is injective. Its image is exactly $\ker(s: \Xi_\rho \mc{O}_U \rightarrow \Xi_\rho \mc{O}_U)$. This is immediately apparent in Figure \ref{fig: comparing filtrations}. Hence $\overline{\mathrm{can}}$ provides an explicit isomorphism which can be used to identify the two copies of $j_! \mc{O}_U$. Under this identification, the algebraic Jantzen filtration \ref{eq: algebraic Jantzen on global sections} and the geometric Jantzen filtration \ref{eq: global sections of geometric Jantzen} clearly align.

\bibliographystyle{alpha}

\bibliography{Jantzen}

\begin{thebibliography}{AvLTV20}

\bibitem[AvLTV20]{ALTV}
J.~D. Adams, M.~A. van Leeuwen, P.~E. Trapa, and D.~A. Vogan, Jr.
\newblock Unitary representations of real reductive groups.
\newblock {\em Ast\'{e}risque}, (417):viii + 188, 2020.

\bibitem[Bar83]{Barbasch}
D.~Barbasch.
\newblock Filtrations on {V}erma modules.
\newblock In {\em Annales scientifiques de l'{\'E}cole Normale Sup{\'e}rieure},
  volume~16, pages 489--494, 1983.

\bibitem[BB81]{BB81}
A.~Beilinson and J.~Bernstein.
\newblock Localisation de g-modules.
\newblock {\em CR Acad. Sci. Paris}, 292:15--18, 1981.

\bibitem[BB93]{BBJantzen}
A.~Beilinson and J.~Bernstein.
\newblock A proof of {J}antzen conjectures.
\newblock In {\em I. {M}. {G}el'fand {S}eminar}, volume~16 of {\em Adv. Soviet
  Math.}, pages 1--50. Amer. Math. Soc., Providence, RI, 1993.

\bibitem[Bei87]{Bei87}
A.~A. Beilinson.
\newblock {\em How to glue perverse sheaves}, pages 42--51.
\newblock Springer Berlin Heidelberg, Berlin, Heidelberg, 1987.

\bibitem[BG99]{BG99}
A.~Beilinson and V.~Ginzburg.
\newblock Wall-crossing functors and $\mc{D}$-modules.
\newblock {\em Representation Theory of the American Mathematical Society},
  3(1):1--31, 1999.

\bibitem[CG97]{chrissginzburg}
N.~Chriss and V.~Ginzburg.
\newblock {\em Representation theory and complex geometry}, volume~42.
\newblock Springer, 1997.

\bibitem[Del80]{Del}
P.~Deligne.
\newblock La conjecture de {Weil} : {II}.
\newblock {\em Publications Math\'ematiques de l'IH\'ES}, 52:137--252, 1980.

\bibitem[GJ81]{GJ}
O.~Gabber and A.~Joseph.
\newblock Towards the {K}azhdan-{L}usztig conjecture.
\newblock {\em Ann. Sci. \'Ecole Norm. Sup. (4)}, 14(3):261--302, 1981.

\bibitem[HTT08]{HoTa}
R.~Hotta, K.~Takeuchi, and T.~Tanisaki.
\newblock {\em D-Modules, Perverse Sheaves, and Representation Theory}.
\newblock Progress in Mathematics. Birkh{\"a}user, 2008.

\bibitem[Hum08]{BGGcatO}
J.~E. Humphreys.
\newblock {\em Representations of semisimple {L}ie algebras in the {BGG}
  category $\mathcal{O}$}, volume~94 of {\em Graduate Studies in Mathematics}.
\newblock American Mathematical Society, Providence, RI, 2008.

\bibitem[IK11]{IK11}
K.~Iohara and Y.~Koga.
\newblock {\em Representation Theory of the Virasoro Algebra}.
\newblock Springer Monographs in Mathematics. Springer London, 2011.

\bibitem[Jan79]{Jantzen}
J.~C. Jantzen.
\newblock {\em Moduln mit einem h\"ochsten {G}ewicht}, volume 750 of {\em
  Lecture Notes in Mathematics}.
\newblock Springer, Berlin, 1979.

\bibitem[KL79]{KL}
D.~Kazhdan and G.~Lusztig.
\newblock Representations of {C}oxeter groups and {H}ecke algebras.
\newblock {\em Inventiones mathematicae}, 53(2):165--184, 1979.

\bibitem[K{\"u}b12a]{Kubel1}
J.~K{\"u}bel.
\newblock From {J}antzen to {A}ndersen filtration via tilting equivalence.
\newblock {\em Math. Scand.}, 110(2):161--180, 2012.

\bibitem[K{\"u}b12b]{Kubel2}
J.~K{\"u}bel.
\newblock Tilting modules in category {$\mc{O}$} and sheaves on moment graphs.
\newblock {\em J. Algebra}, 371:559--576, 2012.

\bibitem[LS06]{LS06}
T.~Levasseur and J.~T. Stafford.
\newblock Differential operators and cohomology groups on the basic affine
  space.
\newblock In {\em Studies in {L}ie theory}, volume 243 of {\em Progr. Math.},
  pages 377--403. Birkh\"{a}user Boston, Boston, MA, 2006.

\bibitem[Mil]{D-modulesnotes}
D.~Mili{\v{c}}i{\'c}.
\newblock Lectures on {A}lgebraic {T}heory of {D}-{M}odules.
\newblock Unpublished manuscript available at
  \url{http://math.utah.edu/~milicic}.

\bibitem[MP98]{MP}
D.~Mili{\v{c}}i{\'{c}} and P.~Pand{\v{z}}i{\'{c}}.
\newblock {\em Equivariant Derived Categories, {Z}uckerman Functors and
  Localization}, pages 209--242.
\newblock Birkh{\"a}user Boston, 1998.

\bibitem[Rom21]{Rom21}
A.~Romanov.
\newblock Four examples of beilinson--bernstein localization.
\newblock {\em Lie Groups, Number Theory, and Vertex Algebras}, 768:65--85,
  2021.

\bibitem[Sai88]{Saito1}
M.~Saito.
\newblock Modules de {H}odge polarisables.
\newblock {\em Publications of the Research Institute for Mathematical
  Sciences}, 24(6):849--995, 1988.

\bibitem[Sai90]{Saito2}
M.~Saito.
\newblock Mixed {H}odge modules.
\newblock {\em Publications of the Research Institute for Mathematical
  Sciences}, 26(2):221--333, 1990.

\bibitem[Sha72]{Shapovalov}
N.~N. Shapovalov.
\newblock On a bilinear form on the universal enveloping algebra of a complex
  semisimple {L}ie algebra.
\newblock {\em Functional Analysis and Its Applications}, 6(4):307--312, 1972.

\bibitem[Soe08]{Soergel08}
W.~Soergel.
\newblock Andersen filtration and hard {L}efschetz.
\newblock {\em Geometric and Functional Analysis}, 17(6):2066--2089, 2008.

\bibitem[Wil16]{Williamson}
G.~Williamson.
\newblock Local {H}odge theory of {S}oergel bimodules.
\newblock {\em Acta Math.}, 217(2):341--404, 2016.

\end{thebibliography}

\end{document}